\documentclass[10pt]{amsart} 

\usepackage{amsmath, amssymb, mathrsfs, mathtools}

\usepackage[mathscr]{euscript} 

\usepackage[final]{todonotes}
\usepackage[final]{showkeys} 

\newlength{\mylength}
\usepackage{enumitem}
\setenumerate{listparindent=\parindent, itemsep=0cm, parsep=\mylength, topsep=0cm}
\setitemize{listparindent=\parindent, itemsep=0cm, parsep=\mylength, topsep=0cm}

\usepackage[breaklinks=true]{hyperref} 
\usepackage{comment} 

\usepackage{url}

\usepackage{tikz-cd}
\usetikzlibrary{decorations.pathmorphing} 

\usepackage{amsthm}
\usetikzlibrary{arrows}

\makeatletter
\renewenvironment{proof}{\par
	\pushQED{\qed}%
	\normalfont \topsep6\p@\@plus6\p@\relax
	\noindent\emph{Proof.} 
	\ignorespaces
}{%
\popQED\endtrivlist\@endpefalse
}
\makeatother



\newtheoremstyle{mythm}
{\mylength}
{4pt}
{\itshape}
{0pt}
{\bfseries}
{.\ }
{ }
{\thmname{#1}\thmnumber{ #2}\thmnote{ (#3)}}

\newtheoremstyle{myrmk}
{\mylength}
{4pt}
{}
{0pt}
{\itshape}
{.\ }
{ }
{\thmname{#1}\thmnumber{ #2}\thmnote{ (#3)}}

\newtheoremstyle{mydef}
{\mylength}
{4pt}
{}
{0pt}
{\bfseries}
{.\ }
{ }
{\thmname{#1}\thmnumber{ #2}\thmnote{ (#3)}}

\theoremstyle{mythm} 

\newtheorem{theorem}[subsubsection]{Theorem}

\newtheorem{question}{Question}
\newtheorem{conjecture}{Conjecture}

\newtheorem{prop}[subsubsection]{Proposition} 
\newtheorem{lem}[subsubsection]{Lemma}
\newtheorem*{lem*}{Lemma}
\newtheorem{cor}[subsubsection]{Corollary}

\theoremstyle{mydef}
\newtheorem{defn}[subsubsection]{Definition} 

\theoremstyle{myrmk} 
\newtheorem{rmk}[subsubsection]{Remark}
\newtheorem*{rmk*}{Remark}
\newtheorem{ex}[subsubsection]{Example}

\usepackage[breaklinks=true]{hyperref} 
\usepackage{comment} 

\usepackage{url}

\usepackage{tikz-cd}

\newcommand{\nc}{\newcommand} 
\nc{\on}{\operatorname}
\nc{\rnc}{\renewcommand}

\nc{\wt}{\widetilde}
\nc{\wh}{\widehat} 
\nc{\ol}{\overline} 

\nc{\BO}{\mathbb{O}}
\nc{\BN}{\mathbb{N}}
\nc{\BZ}{\mathbb{Z}}
\nc{\BQ}{\mathbb{Q}}
\nc{\BR}{\mathbb{R}}
\nc{\BC}{\mathbb{C}}
\nc{\BA}{\mathbb{A}}
\nc{\BP}{\mathbb{P}}
\nc{\BF}{\mathbb{F}}
\nc{\BG}{\mathbb{G}}
\nc{\ul}{\underline}

\nc{\pr}{\on{pr}}
\nc{\id}{\on{id}}
\nc{\Id}{\on{Id}}
\nc{\Tr}{\on{Tr}}

\nc{\Coh}{\on{Coh}}
\nc{\Pic}{\mathrm{Pic}}
\nc{\curAut}{\mathcal{A}\kern -.5pt ut}

\nc{\la}{\langle}
\nc{\ra}{\rangle} 
\nc{\lV}{\lVert}
\nc{\rV}{\rVert}

\nc{\mb}{\mathbf}
\nc{\mf}{\mathfrak}
\nc{\mc}{\mathscr}
\rnc{\mathcal}{\mathscr}

\nc{\hra}{\hookrightarrow}
\nc{\sra}{\twoheadrightarrow} 

\nc{\coker}{\on{coker}}
\nc{\End}{\on{End}}
\nc{\Ext}{\on{Ext}}
\nc{\Spec}{\on{Spec}}
\nc{\Proj}{\on{Proj}}
\rnc{\Im}{\on{Im}}

\nc{\Hom}{\on{Hom}}
\nc{\curHom}{\mathcal{H}\kern -.5pt om}
\nc{\curExt}{\mathcal{E}\kern -.5pt xt}
\nc{\EExt}{\on{\mathbb{E}\kern -.5pt \mathrm{xt}}}
\nc{\curEnd}{\mathcal{E}\kern -.5pt nd}
\nc{\curSym}{\mathcal{S}\kern -.5pt ym}
\nc{\curSpec}{\mathcal{S}\kern -0.5pt pec}
\nc{\curProj}{\mathcal{P}\kern -0.5pt roj}
\DeclareMathOperator*\colim{colim}

\nc{\oh}{\mc{O}}
\nc{\Vect}{\mathrm{Vect}}
\nc{\op}{\mathrm{op}}

\rnc{\setminus}{\smallsetminus}
\nc{\acts}{\curvearrowright}
\nc{\lacts}{\curvearrowleft}
\nc{\act}{\on{act}}

\nc{\SL}{\on{SL}}
\nc{\GL}{\on{GL}}
\nc{\gm}{\mathbb{G}_m}
\nc{\ga}{\mathbb{G}_a}
\nc{\D}{\mc{D}}
\nc{\gr}{\on{gr}}
\nc{\triv}{\on{triv}}
\nc{\QCoh}{\on{QCoh}}
\nc{\tf}{\mathrm{tf}}
\nc{\Com}{\on{Com}}
\nc{\pt}{\mathrm{pt}}
\nc{\Rep}{\on{Rep}}
\nc{\stab}{\on{Stab}}
\nc{\Ad}{\on{Ad}}
\nc{\I}{\mc{I}}
\nc{\back}{\backslash}
\nc{\tame}{{\on{tame}}}
\nc{\Sch}{\on{Sch}}
\nc{\locsys}{\on{LocSys}}
\nc{\Char}{\on{Char}}
\nc{\Perv}{\on{Perv}}

\newenvironment{cd}{\begin{equation*}\begin{tikzcd}}{\end{tikzcd}\end{equation*}\ignorespacesafterend}
\nc{\e}[1]{\begin{align*} #1 \end{align*}}

\usepackage[margin=1.5in]{geometry}

\nc{\collapse}{\vspace{-\mylength}}

\makeatletter
\def\blfootnote{\gdef\@thefnmark{}\@footnotetext}
\makeatother

\title{Equivariant vector bundles on varieties with codimension-one orbits}  

\author{Lucas Mason-Brown and James Tao}

\begin{document}
	
	\maketitle
	
		\begin{abstract}
	    Let $G$ be an algebraic group and let $X$ be a smooth $G$-variety with two orbits: an open orbit and a a closed orbit of codimension $1$. We give an algebraic description of the category of $G$-equivariant vector bundles on $X$ under a mild technical hypothesis. We deduce simpler classifications in the special cases of line bundles and vector bundles which are generically local systems. We apply our results to the study of admissible representations of semisimple Lie groups. Our main result gives a new set of constraints on the associated cycles of unipotent representations.
	\end{abstract}

	\tableofcontents

	\section{Introduction}

	\subsection{Summary} \label{ss:summary}
	
	Work over an algebraically closed field $k$ of characteristic zero. Consider a variety $X$ acted on by an algebraic group $G$. Make the following assumptions: 
	\begin{itemize}
		\item $X$ is smooth. 
		\item $X$ consists of one open $G$-orbit and one closed $G$-orbit of codimension one. 
	\end{itemize}
	Then our main result (Theorem~\ref{thm-main}) gives an algebraic description of the category of $G$-equivariant vector bundles on $X$, provided that $X$ satisfies a technical hypothesis called `fastenedness' (Definition~\ref{def:fastened}). 
	
	In the rest of the paper, we deduce simpler classifications in the special cases of line bundles (Proposition~\ref{prop:linebundles}) and vector bundles which are generically $G$-equivariant local systems (Corollary~\ref{cor-locsys}). We also use the main result to describe $G$-equivariant vector bundles on \emph{weakly normal} $G$-varieties obtained by gluing together those of the form described above (Section~\ref{sec:wn}). As an application, we study the coherent sheaves which arise as associated graded modules of $(\mf{g}, K)$-modules, when the pair $(\mf{g}, K)$ arises from a real reductive group (Section~\ref{sec:application}). 
	
	In the rest of this introduction, we discuss Theorem~\ref{thm-main} in the larger context of `gluing of categories,' and we describe some of these applications in more detail.
	
	\subsection{The idea of gluing}  \label{intro-gluing}
	
	Let $X$ be as the first paragraph of~\ref{ss:summary}, and write $X = U \cup Z$ where $U$ is the open $G$-orbit and $Z$ is the closed $G$-orbit of codimension one, equipped with the reduced subscheme structure. Upon choosing basepoints $z \in Z$ and $u \in U$, it is easy to describe the categories of $G$-equivariant vector bundles on $Z$ and $U$ separately: 
	\e{
		\Vect^G(Z) &\simeq \Rep(G^z) \\
		\Vect^G(U) &\simeq \Rep(G^u), 
	}
	where $G^z$ and $G^u$ are the stabilizer groups. It is natural to ask whether $\Vect^G(X)$ can be constructed by `gluing' the two categories displayed above. 
	
	Motivation for this idea comes from Beilinson's description~\cite{Beilinson1987} of perverse (constructible) sheaves on a scheme $X$ which is decomposed as $U \cup Z$ where $Z = V(f)$ for some regular function $f \in \oh_X$. That description says that a perverse sheaf $\mc{F}$ on $X$ can be encoded as a quadruple $(\mc{F}_U, \mc{F}_Z, u, v)$ such that
	\begin{itemize}
		\item $\mc{F}_U$ and $\mc{F}_Z$ are perverse sheaves on $U$ and $Z$, respectively. 
		\item $u$ and $v$ are maps 
		\[
			\Psi_f(\mc{F}_U) \xrightarrow{u} \mc{F}_Z \xrightarrow{v} \Psi_f(\mc{F}_U)
		\]
		where $\Psi_f(-)$ is the nearby cycles functor. (We have ignored the Tate twist.) The composition $u \circ v$ must equal $1 - (\text{monodromy endomorphism})$. 
	\end{itemize}
	This description has the good feature that it involves only the categories of interest on $U$ and $Z$, along with a functor between them (and an endomorphism of this functor). 
	
	For vector bundles, the situation is not as nice. The standard approaches to gluing quasicoherent sheaves (e.g.\ recollement of categories) require knowing about sheaves on $X_{\wh{Z}}$ (the formal completion of $X$ along $Z$). For example, Beauville--Laszlo gluing~\cite{bl} implies that a $G$-equivariant vector bundle on $X$ can be encoded as a triple $(\mc{F}_U, \mc{F}_{\wh{Z}}, \alpha)$ where 
	\begin{itemize}
		\item $\mc{F}_U$ and $\mc{F}_{\wh{Z}}$ are $G$-equivariant vector bundles on $U$ and $X_{\wh{Z}}$, respectively. 
		\item $\alpha$ is an isomorphism between their restrictions to $X_{\wh{Z}} \setminus Z$. 
	\end{itemize} 
	Since $X_{\wh{Z}}$ is not easy to compute with in practice, we sought a different description of $\Vect^G(X)$ which does not involve $X_{\wh{Z}}$. 
	
	In some special cases, knowledge about $X_{\wh{Z}}$ can be gleaned from looking at the action of $G$ on the normal bundle $\mc{N}_{Z/X}$. For instance, if $G$ is reductive and $X$ is affine, then the Luna Slice Theorem provides a one-dimensional $G^z$-invariant locally closed subvariety $\ell \subset X$ which contains $z$ and is transverse to $Z$, together with a $G^z$-equivariant \'etale map $\ell \to \mc{N}_{Z/X}|_z$. Combined with the previous paragraph, this essentially reduces the problem to studying $G$-equivariant vector bundles on $\mc{N}_{Z/X}$. However, in our desired application $X$ is usually not affine, so this reduction does not apply. 
	
	In our desired application, something weaker is true (Proposition~\ref{prop:chXfastened}): there is a line $\ell \subset X$ which intersects $Z$ transversely, is invariant under a one-parameter subgroup $\gamma : \gm \hra G^z$, and is acted upon nontrivially by $\gamma(\gm)$. The existence of such a pair $(\ell, \gamma)$ is equivalent to the `fastenedness' hypothesis of Theorem~\ref{thm-main}. 
	
	Starting from a choice of $(\ell, \gamma)$ as above, we arrive at a description of $\Vect^G(X)$ which bears little resemblance to `gluing.' Instead, it should be thought of in the following terms. If we pull back $\mc{F} \in \Vect^G(X)$ along $\ell \hra X$, two algebraic structures emerge:
	\begin{itemize}
		\item $\mc{F}|_{\ell}$ is equivariant with respect to a $\gm$-action which comes from the subgroup $\gamma$. Recall that the category of $\gm$-equivariant vector bundles on $\BA^1$ is equivalent to the category of finite-dimensional vector spaces equipped with exhausting filtrations, via the Artin--Rees construction (Proposition~\ref{ar}).
		\item Each fiber of $\mc{F}|_{\ell}$ is acted on by the stabilizer group of the corresponding point in $\ell$. 
	\end{itemize}
	The classifying data defined in Theorem~\ref{thm-main} are a combination of these two bullet points. 
	
	However, in the case of line bundles, it is possible to reformulate the main result in terms which more closely resemble `gluing via nearby cycles' see Lemma~\ref{lem-nearby}. The possibility of doing so reflects that equivariant quasicoherent sheaves are more `rigid' than arbitrary quasicoherent sheaves. 
	
	\subsection{Applications} 
	
	Suppose $G_{\mathbb{R}}$ is a real reductive group and $K_{\mathbb{R}} \subset G_{\mathbb{R}}$ is a maximal compact subgroup. Let $K$ be the complexification of $K_{\mathbb{R}}$ and let $\mathcal{N}$ be the cone of nilpotent elements in the complexified Lie algebra of $G_{\mathbb{R}}$
	
	The determination of the irreducible unitary representations of $G_{\mathbb{R}}$ is one of the major unsolved problem in representation theory. There is evidence to suggest that every such representation can be constructed through a sequence of well-understood operations from a finite set of building blocks, called the \emph{unipotent representations}. These representations are `attached' (in a certain mysterious sense) to the nilpotent orbits of $G_{\mathbb{R}}$ on the dual space of its Lie algebra.

	Attached to every irreducible admissible representation $V$ of $G_{\mathbb{R}}$ (for example, every irreducible unitary representation) is a Harish-Chandra module $M$, which has the structure of an irreducible $K$-equivariant $\mathfrak{g}$-module. Taking its associated graded gives a well-defined class $[\gr(M)]$ in the Grothendieck group of $K$-equivariant coherent sheaves on $\mathcal{N}$. This class provides valuable information about $V$, including, for example, its $K_{\mathbb{R}}$-types. Determining the classes $[\gr(M)]$ associated to unipotent representations is a fundamental problem with origins in the seminal paper (\cite{Vogan1991}) of Vogan. 
	
	In Section~\ref{sec:application}, we define the $K$-chain $\overline{\mathrm{Ch}}(M)$ associated to $M$. This is the variety with $K$-action obtained by first deleting the orbits of codimension $\ge 2$ from the reduced support of $\gr(M)$ and then passing to the projectivization (see~\ref{gk-chain} for details). The class $[\gr(M)]$ gives rise to a class $[\overline{\gr}(M)]$ on $\overline{\mathrm{Ch}}(M)$, and when $V$ is unipotent, we conjecture that $[\overline{\mathrm{gr}}(M)]$ determines $[\gr(M)]$. Our main result in Section \ref{sec:application} is that $\overline{\mathrm{Ch}}(M)$ is fastened (Proposition \ref{prop:chXfastened}). To prove this remarkable fact, we use in an essential way the special structure of $\mathcal{N}$ (e.g. the existence of `real' Slodowy slices with nice contracting $\gm$-actions).  Since $\overline{\mathrm{Ch}}(M)$ is fastened, we can  understand $\overline{\mathrm{gr}}(M)$ (and by extension $\gr(M)$ and $V$) using Theorem \ref{thm-main} (and its corollaries).

	\subsection{Acknowledgments} 
	
	The authors would like to thank Roman Bezrukavnikov and David Vogan for many helpful discussions. The second author is supported by an NSF GRFP grant.

	\section{Chains} \label{sec:chains}

    \subsection{Basic definitions} \label{s-def-chains}
    Let $k$ be an algebraically closed field of characteristic zero, and let $G$ be an affine algebraic group over $k$.
	
	\begin{defn}\label{def:chain}
		A $G$-chain is a $k$-variety\footnote{For us, a `variety' is a reduced, separated, finite type $k$-scheme. It is not necessarily irreducible.} $X$ equipped with an action by $G$ such that
		\begin{enumerate}
			\item\label{cond:chain1} $G$ acts on $X$ with finitely many orbits. 
			\item\label{cond:chain2} $X$ is pure dimensional. 
			\item\label{cond:chain3} Each $G$-orbit on $X$ has codimension $0$ or $1$. 
		\end{enumerate}
	\end{defn}
	
	These properties imply that every $G$-orbit on $X$ is either open or closed. Each open orbit $U \subset X$ has codimension 0, and each closed orbit $Z \subset X$ has codimension 1. 
	
	\begin{ex} \label{ex-X}
		Let $G = \gm$, and $X = \Spec \mathbb{C}[x,y]/(xy)$. Let $G$ act on $X$ by the formulas
		$$t \cdot x = t^2x \qquad t \cdot y = t^{-2} y  \qquad t \in \gm$$
		Then $X$ is a $G$-chain. It has two open $G$-orbits (the punctured axes), and one closed $G$-orbit ($\{0\}$). This example is closely related to the (infinite-dimensional) representation theory of $\on{SL}_2(\mathbb{R})$, see Example~\ref{ex-su}. 
	\end{ex}
	
	We say that a $G$-chain is \emph{irreducible} if it has only one open $G$-orbit. (Note that an irreducible $G$-chain need not be irreducible as a $k$-variety if $G$ is not connected.) For an arbitrary $G$-chain $X$, we define its \emph{irreducible components} to be the closures of the open orbits of $X$, equipped with the reduced subscheme structures. These are precisely the irreducible closed sub-chains of $X$ whose dimensions equal $\dim X$. We say that a $G$-chain is \emph{simple} if it is irreducible and has at most one closed $G$-orbit. If $X$ is irreducible, then the simple locally-closed sub-chains form an open cover of $X$. 
	
	If $X$ is a $G$-chain, then its normalization $\tilde{X}$ is again a $G$-chain. Any normal $G$-chain is smooth, since its singular locus is a subset of codimension $\ge 2$ and hence empty. Any smooth $G$-chain is the disjoint union of its irreducible components. In particular, the normalization of an irreducible $G$-chain is a smooth, irreducible $G$-chain. 
	
	\subsection{Fastened chains} 
	
	\begin{defn}\label{def:fastened}
		We define the property of `fastenedness' for $G$-chains in three steps: 
		\begin{itemize}
			\item If $X$ is a smooth irreducible $G$-chain, we say that a closed $G$-orbit $Z \subset X$ is \emph{fastened} to the open orbit $U \subset X$ if $G$ acts transitively on the punctured normal bundle 
			\[
				\mc{N}^{\circ}_{Z / X} := \mc{N}_{Z / X} \setminus (\text{zero section}). 
			\]
			\item If $X$ is a smooth irreducible $G$-chain, we say that $X$ is \emph{fastened} if each closed $G$-orbit $Z \subset X$ is fastened to the open orbit. 
			\item If $X$ is any $G$-chain, we say that $X$ is \emph{fastened} if each irreducible component of the normalization of $X$ is fastened. 
		\end{itemize}
		The terminology is set up so that it is clear what it means for a closed orbit $Z \subset X$ to be fastened to one open orbit which is adjacent to $Z$ but not to another. 
	\end{defn}
	
	\begin{rmk} \label{rmk-fastened-normal}
	Suppose that $X$ is a smooth irreducible $G$-chain and that $Z \subset X$ is a closed $G$-orbit. For any closed point $z \in Z$, the stabilizer subgroup $G^z \subset G$ acts linearly on the (one-dimensional) fiber $\mc{N}_{Z_i / X}|_{z}$. This action defines a character $G^z \to \gm$, which is surjective if and only if $Z$ is fastened to the open orbit of $X$. 
	\end{rmk}
	
	\begin{prop}\label{prop:affinefastened}
		Every affine $G$-chain is fastened.
	\end{prop}
	\begin{proof}
		Let $X$ be an affine $G$-chain. Since normalization maps of finite type schemes are finite, hence affine, we may assume that $X$ is also smooth and irreducible. Pick an arbitrary closed $G$-orbit $Z \subset X$. We shall show that $Z$ is fastened to the open orbit of $X$. 
		
		Choose a point $z \in Z$. Assume for sake of contradiction that the character $G^z \to \gm$ defined in Remark~\ref{rmk-fastened-normal} is not surjective. Then it factors through some finite subgroup $\mu_r \subset \gm$ where $\mu_r$ is the group of $r$-th roots of unity. Choose a linear coordinate $\mc{N}_{Z/X}|_z \simeq \BA^1_t$, and define the regular function 
		\[
			f : \mc{N}_{Z/X} \to \BA^1
		\]
		via the formula $f(t) = t^r$. 
		
		Since $\mc{N}_{Z/X}$ is the associated bundle produced by the homogeneous space $G / G^z \simeq Z$ and the action $G^z \acts \mc{N}_{Z/X}|_z$, the function $f$ defined above extends to a unique $G$-invariant function $\tilde{f}$ on $\mc{N}_{Z/X}$. By construction, it is nonconstant along the fibers of $\mc{N}_{Z/X}$. 
		
		Let 
		\[
			X_h := \on{Bl}_{Z \times \{0\}}(X \times \BA^1_h) \setminus (X \times \{0\})^{\sim} \to \BA^1_h
		\]
		be the deformation to the normal cone of $Z \subset X$. (This construction appears again in~\ref{I}.) Since $X$ is affine, so is $X_h$. Since $G$ is reductive, the functor of $G$-invariants is exact, so we may non-uniquely extend $\tilde{f}$ to a $G$-invariant regular function $\tilde{f}_h$ on $X_h$. Then $\tilde{f}_h$ is constant along the $G$-orbits $U \times \{a\} \subset X_h$ and hence along their limit, which is $\mc{N}_{Z/X}^{\circ}$.\footnote{This argument uses that $X_h \to \BA^1_h$ is flat.} But $\tilde{f}_h|_{\mc{N}_{Z/X}^\circ} = \tilde{f}$ by construction, so this constancy contradicts the previous paragraph. 
	\end{proof}
	
	Outside the world of affine varieties, there are many examples of non-fastened chains. Here is one:
	
	\begin{ex}
		Let $G = SL_2(\mathbb{C})$ and let $V$ be its four-dimensional irreducible representation. $V$ is identified with the space of homogeneous degree-3 polynomials in the variables $x$ and $y$.
		
		Let $U$ be the 3-dimensional $G$-orbit passing through the element $x^2y \in V$. There are two $G$-orbits in its boundary: $\{0\}$ and the 2-dimensional $G$-orbit $Z$ passing through the element $z=x^3 \in V$. Let $X$ be the irreducible chain $U \cup Z$. Note that
		\[
		    G^z = \left\{\begin{pmatrix} \xi & a \\ 0 & \xi^{-1}\end{pmatrix}: \xi^3=1\right\}. 
	    \]
		The variety $X$ is singular, but its normalization $\tilde{X}$ is smooth. For any closed orbit $\tilde{Z} \subset \tilde{X}$, the natural map $\tilde{Z} \to Z$ is a $G$-equivariant cover. In particular, for any $\tilde{z} \in \tilde{Z}$, we have an inclusion of stabilizer groups 
		\[
		    G^{\tilde{z}} \subseteq G^z =  \left\{\begin{pmatrix} \xi & a \\ 0 & \xi^{-1}\end{pmatrix}: \xi^3=1\right\}. 
		\]
		Thus $G^{\tilde{z}}$ does not contain a torus, so Remark~\ref{rmk-fastened-normal} implies that $X$ is not fastened.
	\end{ex}

	\subsection{Fastening data}
	We will show that fastened chains are glued together (or `fastened') by certain $\gm$-invariant curves, as described below: 

	\begin{defn}\label{def:fasteningdatum}
	Let $X$ be a smooth irreducible $G$-chain. Let $U$ be the open orbit of $X$, and let $Z$ be some closed orbit. 
		A \emph{fastening datum} for $(U, Z)$ is a pair $(\gamma, \ell)$ given as follows: 
		\begin{itemize}
			\item $\gamma: \gm \to G$ is a co-character. 
			\item $\ell: \BA^1 \hra X$ is a locally-closed embedding. 
		\end{itemize}
		We require $(\gamma, \ell)$ to satisfy the following properties: 
		\begin{enumerate}
			\item $\ell^{-1}(Z) = \{0\}$ and $\ell$ is transverse to $Z$. 
			\item $\gamma(\gm)$ leaves $\ell(\BA^1) \subset X$ invariant and acts transitively on $\ell(\BA^1 \setminus \{0\})$. 
		\end{enumerate}
	\end{defn}
	
	\begin{prop}\label{prop:fasteningdata}
		In the setting of Definition \ref{def:fasteningdatum}, $Z$ is fastened to $U$ if and only if there is a fastening datum $(\gamma,\ell)$ for $(U,Z)$.
	\end{prop}
	
	\begin{proof}
		First, assume that $Z$ is fastened to $U$. Fix a closed point $z \in Z$. By Remark~\ref{rmk-fastened-normal}, the action $G^z \acts \mc{N}_{Z/X}|_z$ occurs via a surjective character $G^z \xrightarrow{\chi} \gm$. Let $L^z \subset G^z$ be the identity component of a Levi subgroup of $G$. Since $\ker(\chi)$ contains the unipotent radical of $G^z$, the surjectivity of $\chi$ implies that $L^z \xrightarrow{\chi} \gm$ is also surjective. Since one-dimensional tori are dense in reductive groups, there is a one-dimensional torus $T^z \subset L^z$ such that $T^z \xrightarrow{\chi} \gm$ is nontrivial (hence surjective). Since $T^z$ is reductive, the projection map $\mc{T}_zX \to \mc{N}_{Z/X}|_z$ has a $T^z$-invariant section. The image of this section is a line $k \cdot v \subset \mc{T}_z X$ on which $T^z$ acts by the same nontrivial character $\chi$. 
		
		By Sumihiro's theorem on torus actions, which applies since $X$ is smooth and hence normal, there exists a $T^z$-invariant open affine subscheme $W_z \subset X$ which contains $z$. Applying the Luna Slice Theorem to the $T^z$-orbit given by $\{z\} \subset W_z$, we conclude (possibly upon replacing $W_z$ by a smaller $T^z$-invariant affine open subscheme) that there is a $T^z$-equivariant \'etale map $\varphi : W_z \to \mc{T}_zX$ which induces an isomorphism on tangent spaces at $z \in W_z$. Consider the base change 
		\begin{cd}[column sep = 0.6in]
			C \arrow[r,hookrightarrow] \arrow[d] & W_z \arrow[d,"\varphi"]\\
			k \cdot v \arrow[r,hookrightarrow, "\text{cl.emb.}"] & \mc{T}_zX 
		\end{cd}
		where the bottom horizontal map was constructed in the previous paragraph. Certainly $C$ contains $z \in W_z$ because $k \cdot v$ contains $0 \in \mc{T}_z X$. Let $C_z \subset C$ be the connected component of $C$ which contains $z$. The map $C_z \to k \cdot v$ is \'etale and $\gm$-equivariant, so the previous two sentences imply that it is an isomorphism.\footnote{Proof: Since $C_z$ is a smooth affine connected curve equipped with a $\gm$-action with a closed orbit $\{z\}$, it is isomorphic to $\BA^1$ with some multiple of the standard $\gm$-action. Since the map $C_z \to k \cdot v$ induces an isomorphism on tangent spaces at $z \in C_z$, this multiple must be the same as that given by $\chi|_{T^z}$. Now the map $C_z \to k \cdot v$ is determined by its behavior on the open $\gm$-orbits; there it is an isomorphism, and one can easily check that this implies the claim.}
		
		We define $\ell$ to be the composition 
		\[
			\BA^1 \simeq k \cdot v \simeq C_z \xhookrightarrow{\text{cl.emb.}} W_z \xhookrightarrow{\text{open emb.}} X, 
		\]
		and we define $\gamma$ to be the composition
		\[
			\gm \simeq T^z \hra G^z \hra G
		\]
		where $\gm \simeq T^z$ is an arbitrary isomorphism. Then $(\gamma, \ell)$ is a fastening datum for $(U, Z)$. 
		
		Conversely, assume that $(\gamma, \ell)$ is a fastening datum for $(U, Z)$.  Definition~\ref{def:fasteningdatum}(2) implies that the map $\BA^1 \xrightarrow{\ell} X$ is $\gm$-equivariant with respect to some nontrivial linear action $\gm \acts \BA^1$ (which is not necessarily the standard one). Passing to tangent spaces, we get a $\gm$-equivariant map 
		\[
			\BA^1 \simeq \mc{T}_0 \BA^1 \xrightarrow{d\ell} \mc{T}_zX
		\]
		where the action of $\gm \acts \mc{T}_zX$ is via the cocharacter $\gamma : \gm \to G^z$ (this uses Definition~\ref{def:fasteningdatum}(1)). By the transversality statement in  Definition~\ref{def:fasteningdatum}(2), the composition 
		\[
			\BA^1 \simeq \mc{T}_0 \BA^1 \xrightarrow{d\ell} \mc{T}_zX \to \mc{N}_{Z/X}|_z
		\]
		is nonzero, hence an isomorphism. It is also $\gm$-equivariant, so we conclude that the action $\gamma(\gm) \acts \mc{N}_{Z/X}|_z$ is nontrivial. Hence, the action $G^z \acts \mc{N}_{Z/X}|_z$ is nontrivial. By Remark~\ref{rmk-fastened-normal}, this implies that $Z$ is fastened to the open orbit in $X$. 
	\end{proof}

\section{Equivariant sheaves on smooth simple chains}\label{sec:main}

	\subsection{Overview} 
	In this section, $X = U \cup Z$ is a smooth $G$-chain with one open and one closed orbit. Assume that $X$ is fastened, and let $(\gamma, \ell)$ be a fastening datum (Definition~\ref{def:fasteningdatum}). 
	
	Given this datum, we will describe the category $\QCoh^G_{\tf}(X)$ of $G$-equivariant torsion-free quasicoherent sheaves on $X$ in algebraic terms (Theorem~\ref{thm-main}). The main idea is that the line $\ell : \BA^1 \to X$ functions as a `basepoint' which sees both the open and the closed orbit of $X$. More precisely, we will explain how $X$ is related to the `homogeneous space' given by $G \times \BA^1$ modulo the stabilizer scheme of $\ell(\BA^1)$, see Lemma~\ref{lem-gs}. 
	
	The special case when $G \acts X$ is the standard (weight one) action $\gm \acts \BA^1$ will feature prominently in our analysis. Here, the desired classification is elementary and well-known: 	
	\begin{prop}[Artin--Rees] \label{ar} 
		Let $\mathrm{Vect}^{\mathrm{filt}}$ be the category of vector spaces equipped with exhausting (but not necessarily separating) $\mathbb{Z}$-filtrations. There is a natural equivalence of monoidal categories:
		\[
		\QCoh^{\gm}_{\tf}(\BA^1) \simeq \mathrm{Vect}^{\mathrm{filt}}
		\]
	\end{prop}
	\begin{proof}
		The category $\QCoh^{\gm}(\BA^1)$ is equivalent (via global sections) to the category of graded modules over the algebra $k[x]$, where the generator $x$ is placed in degree 1. Each such module is encoded by a diagram
		\[
		\cdots \to V_i \to V_{i+1} \to \cdots
		\]
		of vector spaces, where the maps are arbitrary. The module is torsion-free if and only if the maps are all injective. In this case, the diagram of vector spaces is equivalent to the datum of a single vector space $V = \colim_i V_i$ equipped with the exhausting filtration $F$ defined by $F^{\ge i}V = V_i$. The reverse direction is clear. 
	\end{proof}
	
	We will often refer to this equivalence as the Artin--Rees construction. Since filtrations already appear here, it is unsurprising that they also appear in the statement of the main classification result (Theorem~\ref{thm-main}). 
	
	\begin{rmk} \label{rmk-monoidal}
		We now explain how we propose to get `algebraic' descriptions of categories of equivariant sheaves. The appearance of the word `monoidal' in the proposition implies that we get analogous equivalences involving algebro-geometric objects defined over $\BA^1 / \gm$. For example, considering algebra objects on both sides yields the equivalence 
		\[
		\on{Sch}^{\on{aff}, \tf}_{\BA^1 / \gm} \simeq (\text{filtered rings})^{\op}
		\]
		where the left hand side is the category of affine schemes over $\BA^1 / \gm$ for which the structure map $S \to \BA^1 / \gm$ is flat. 
		
		\collapse
		
		This equivalence intertwines the Cartesian monoidal structure on the left hand side with the $\otimes$ monoidal structure on the right hand side. Passing to group objects with respect to these monoidal structures, we obtain a (contravariant) equivalence between affine group schemes flat over $\BA^1 / \gm$ and filtered Hopf algebras. This list of equivalences continues indefinitely, but essentially we will only need one more observation: given such a group scheme $\mc{G} \to \BA^1 / \gm$ corresponding to a filtered Hopf algebra $A$, the category of $\gm$-equivariant torsion-free quasicoherent sheaves on the stack $\BA^1 / \mc{G}$ is equivalent to the category of filtered comodules for $A$. This is applied in Definition~\ref{claim1}. 
	\end{rmk}
	
	\subsubsection{Further comments} 
	
	In the above proposition, the `torsion-free' requirement is equivalent to flatness. Similarly, a $G$-equivariant sheaf on $X$ is torsion-free if and only if it is flat, because $Z$ has codimension one. In the rest of this section, we use `torsion-free' and `flat' interchangeably. 
	
	In this section, a (back)slash means a quotient in the stack-theoretic sense. 
	
	To clarify what is canonical and what is not, let us explain where the fastening datum $(\gamma, \ell)$ is used. In~\ref{I}, the fastening datum is not used, but Lemma~\ref{deform} relies on the assumption that $X$ is fastened. In~\ref{slice-line}, the line $\ell$ (satisfying Definition~\ref{def:fasteningdatum}(1)) is used to define the stabilizer scheme $G^s$. From~\ref{slicing-gm} onward, we use the full datum $(\gamma, \ell)$.  
	
	\subsection{The punctured ideal sheaf} \label{I} 
	
	Let $\I_{Z/X}$ be the ideal sheaf of $Z \subset X$, which is a line bundle on $X$. Consider its total space 
	$
	\curSpec (\curSym_{\oh_X}^\cdot \I_{Z/X}^\vee)
	$
	which we (abusively) also denote $\I_{Z/X}$. Let
	\[
	\I^\circ_{Z/X} := \I_{Z/X} \setminus (\text{zero section})
	\]
	be the punctured ideal sheaf. Since $Z$ is $G$-invariant in $X$, there are commuting actions
	\begin{equation*}
	G \acts \I^\circ_{Z/X} \lacts \gm
	\end{equation*}
	where $\gm$ acts by dilation on the fibers. Since $\I^\circ_{Z/X} / \gm \simeq X$, we deduce that 
	\begin{equation}\label{equiv1} 
	\QCoh^G(X) \simeq \QCoh^{G \times \gm}(\I^\circ_{Z/X}) \simeq \QCoh^{\gm}(G \back \I^\circ_{Z/X}). 
	\end{equation}
	Our goal in this and the next subsection is to explain how $G \back \I^\circ_{Z/X}$ is related to the stabilizer scheme of the line $\ell : \BA^1 \to X$. 
	
	\begin{rmk} \label{rmk-I-geo}
		We point out two facts about the geometry of $\I^\circ_{Z/X}$. 
		\begin{enumerate}[label=(\roman*)]
			\item $\I^\circ_{Z/X}$ is closely related to the deformation of $X$ to the normal cone of $Z \subset X$. To see this, note that the injection of sheaves $\I_{Z/X} \hra \oh_X$ defines a (noninjective) map of varieties 
			\[
			\I_{Z/X} \xrightarrow{p} X \times \BA^1. 
			\]
			It is easy to check that this is isomorphic to the map 
			\[
			\on{Bl}_{Z \times \{0\}}(X \times \BA^1) \setminus (Z \times \BA^1)^{\sim} \to X \times \BA^1
			\]
			where the left hand side is obtained by taking the blow-up of $X \times \BA^1$ along $Z \times \{0\}$ and deleting the proper transform of $Z \times \BA^1$. On the other hand, the deformation to the normal cone is given by 
			\[
			\on{Bl}_{Z \times \{0\}}(X \times \BA^1) \setminus (X \times \{0\})^{\sim} \to X \times \BA^1. 
			\]
			This is not the same as the previous map, but the deformation to the normal cone contains the open subscheme
			\[
			\on{Bl}_{Z \times \{0\}}(X \times \BA^1) \setminus \Big((X \times \{0\})^{\sim} \sqcup (Z \times \BA^1)^{\sim}\Big)
			\]
			which is isomorphic to $\I^\circ_{Z/X}$ as schemes over $X \times \BA^1$. 
			\item Note that $\mc{I}_{Z/X}^\circ|_{Z}$ is the punctured conormal bundle of $Z$, while the fiber over $0 \in \BA^1$ of the scheme displayed in the last line above is the punctured normal bundle of $Z$. This apparent contradiction is resolved by the observation that the punctured total spaces of $\mc{L}$ and $\mc{L}^{\vee}$ are canonically isomorphic (but not in a $\gm$-equivariant way), for any line bundle $\mc{L}$. Thus, there is no harm in writing $\mc{I}_{Z/X}^\circ|_Z \simeq \mc{N}_{Z/X}^\circ$.
		\end{enumerate}
	\end{rmk}
	
	Define the `projection' map 
	$
		\pi = \mathrm{pr}_2 \circ p : \I^\circ_{Z/X} \to \BA^1. 
	$
	\begin{lem} \label{deform} 
		We have the following: 
		\begin{enumerate}[label=(\roman*)]
			\item $\pi$ is smooth and $G \times \gm$-equivariant, where the action $(G \times \gm) \acts \BA^1$ is given by $(\on{triv}, \on{standard})$. 
			\item As $G$-varieties, the closed fibers of $\pi$ are given by 
			\[
				\pi^{-1}(c) \simeq \begin{cases}
					U &\text{ if } c \neq 0 \\
					\mc{N}^{\circ}_{Z/X} &\text{ if } c = 0. 
				\end{cases}
			\]
			In particular, $G$ acts transitively on the fibers of $\pi$. 
		\end{enumerate}
	\end{lem}
	\begin{proof}
		The map $\pi$ is $G \times \gm$-equivariant because $p$ and $\pr_2$ are both $G \times \gm$-equivariant. The description of fibers of $\pi$ follows from Remark~\ref{rmk-I-geo} and the transitivity of the $G$-action on fibers follows from the assumption that $X$ is fastened. This proves (ii). 
		
		Since the domain and target of $\pi$ are smooth, and each fiber has dimension equal to $\dim X$, the Miracle Flatness Theorem implies that $\pi$ is flat. Since each fiber is also smooth, this implies that $\pi$ is smooth (since $\pi$ is finitely presented), which proves (i). 
	\end{proof}

	\subsection{The group scheme \texorpdfstring{$G^s$}{Gs} associated to a fastening datum}\label{slice-line}
	
	The choice of a line $\ell : \BA^1 \to X$ which satisfies Definition~\ref{def:fasteningdatum}(1) is equivalent to the choice of a section $s : \BA^1 \to \I^\circ_{Z/X}$ of the projection map $\pi$. The forward direction of this equivalence sends $\ell$ to the proper transform of $(\ell, \id_{\BA^1}) : \BA^1 \hra X \times \BA^1$ under the blow-up considered in Remark~\ref{rmk-I-geo}(i), and the transversality hypothesis is needed to ensure that the proper transform lands in the open subscheme $\I^\circ_{Z/X}$. 
	
	In the rest of this section, we will work in the category of schemes (or stacks) over $\BA^1$. (Starting from the next subsection, these objects will also be $\gm$-equivariant, i.e.\ they descend to $\BA^1 / \gm$.) From this point of view, the section $s$ functions as a `basepoint' for $\mc{I}^\circ_{Z/X}$, which we view as a homogeneous space for the group scheme $G \times \BA^1$. The rest of this subsection makes this statement precise. 
	
	Define the map $a$ via the commutative diagram 
	\begin{cd}
		G \times \BA^1 \ar[d, hookrightarrow, swap, "\id_G \times s"] \ar[rd, "a"] \\
		G \times \I^\circ_{Z/X} \ar[r, "\act"] & \I^\circ_{Z/X}
	\end{cd}
	where $\act$ is the $G$-action on $\I^{\circ}_{Z/X}$. Define $G^s$ as the fibered product 
	\begin{cd}
		G^s \ar[r, hookrightarrow] \ar[d, swap] & G \times \BA^1  \ar[d, "a"] \\
		\BA^1 \ar[r, hookrightarrow, "s"] & \I^\circ_{Z/X}
	\end{cd}	
	
	\begin{lem}\label{lem-gs}
		Work in the category of schemes over $\BA^1$. 
		\begin{enumerate}[label=(\roman*)]
			\item $G^s$ is smooth over $\BA^1$. 
			\item The map $G^s \hra G \times \BA^1$ expresses the former as a group subscheme of the latter. 
			\item The quotient $(G \times \BA^1) / G^s$ taken in the smooth topology (or any finer one) identifies with $\I^\circ_{Z/X}$ via the map $a$. 
		\end{enumerate}
	\end{lem}
	\begin{proof}
		(ii). The indicated map $G^s \hra G \times \BA^1$ is a closed embedding because the map $s$ is a closed embedding. The `group subscheme' property follows from the definition of group action applied to $(G \times \BA^1) \acts \I^\circ_{Z/X}$. 
		
		(i). The Miracle Flatness Theorem implies that $a$ is flat, since the source and target of $a$ are smooth varieties, and Lemma~\ref{deform}(ii) ensures that the fibers have the correct dimensions. Furthermore, $a$ is finitely presented, and its fibers are smooth because they are isomorphic to stabilizer group schemes, and group schemes are smooth in characteristic zero. Therefore $a$ is smooth. Since $G^s \to \BA^1$ is obtained from $a$ by base change, it is smooth as well. 		
		
		(iii). The following statement is well-known in algebraic geometry: \begin{itemize}
			\item Let $S$ be a scheme, and work in the category of sheaves of groupoids on the category of affine schemes over $S$ (denoted $\on{Sch}^{\on{aff}}_{/S}$) with respect to some Grothendieck topology. Let $\mc{G}$ be a group object, and let $\mc{G} \acts \mc{I}$ be a left action on some other object $\mc{I}$. Let $s : S \to \mc{I}$ be a section, let $a : \mc{G} \to \mc{I}$ be the action map, and let $\mc{G}^s$ be the stabilizer subgroup of $s$ (defined as above). If $a$ is a cover, then $\mc{G} / \mc{G}^s \simeq \mc{I}$. 
		\end{itemize}
		This implies (iii) by taking $S = \BA^1$, working in the smooth topology, and taking $\mc{G} = G \times \BA^1$ and $\mc{I} = \mc{I}^\circ_{Z/X}$. Our proof of (i) shows that $a$ is a smooth cover, so the statement applies. 
		
		For sake of completeness, we prove the statement in the bullet point. Since $a$ is a cover, the geometric realization of the Cech nerve of $a$ is isomorphic to $\mc{I}$. The Cech nerve of $a$ is the simplicial object 
		\begin{cd}
			\cdots \ar[r, shift left = 2] \ar[r, shift left = -2] \ar[r] & \mc{G} \underset{\mc{I}}{\times} \mc{G} \ar[r, shift left = 1] \ar[r, shift right = 1] & \mc{G}
		\end{cd}
		It suffices to show that this is isomorphic to the simplicial object which encodes the right action of $\mc{G}^s$ on $\mc{G}$, shown below: 
		\begin{cd}
			\cdots \ar[r, shift left = 2] \ar[r, shift left = -2] \ar[r] & \mc{G} \times \mc{G}^s \ar[r, shift left = 1] \ar[r, shift right = 1] & \mc{G}
		\end{cd}
		At the level of objects, this isomorphism is defined by the maps 
		\[
		\mc{G} \times \mc{G}^s \times \cdots \times \mc{G}^s \to \mc{G} \underset{\mc{I}}{\times} \mc{G} \underset{\mc{I}}{\times} \cdots \underset{\mc{I}}{\times} \mc{G}
		\]
		given by the formula 
		\[
		(g_1, g_2, \ldots, g_n) \mapsto (g_1, g_1g_2, \ldots, g_1 g_2 \cdots g_n),  
		\]
		where $g_1$ is a $T$-point of $\mc{G}$ and $g_2,\ldots, g_n$ are $T$-points of $\mc{G}^s$, for any test object $T \in \on{Sch}^{\on{aff}}_{/S}$. It is straightforward to make these maps compatible with the maps in the simplicial objects. 
	\end{proof}
	
	\begin{rmk} \label{rmk-orbits} 
		For $c \in \BA^1 \setminus \{0\}$, we have $G / G^{s(c)} \simeq U$ where $G^{s(c)}$ means the stabilizer of $s(c) \in U$. For $0 \in \BA^1$, we similarly have $G / G^{s(0)} \simeq \mc{N}^{\circ}_{Z/X}$. The preceding lemma smoothly interpolates between these two cases, since $G^s|_c \simeq G^{s(c)}$.  
	\end{rmk}
	
	In~\ref{I}(\ref{equiv1}), replacing $\I^\circ_{Z/X}$ by $(G \times \BA^1) / G^s$ yields a preliminary result:  
	\[
	\QCoh^G(X) \simeq \QCoh^{\gm}(\BA^1 / G^s). 
	\]
	The category $\QCoh(\BA^1 / G^s)$ is just the category of sheaves on $\BA^1$ equipped with an action by the group scheme $G^s$. The problem is that the action $\BA^1 / G^s \lacts \gm$ is inexplicit. It does \emph{not} arise from an action $G^s \lacts \gm$ via automorphisms of $G^s$ as a group scheme over $\BA^1$. We will rectify this problem in~\ref{slicing-gm}. 
	
	\subsection{Descent along a degree-\texorpdfstring{$n$}{n} cover} \label{descent} 
	
	Let $\BA^1_h$ be the affine line with coordinate $h$, let $n$ be a positive integer, and consider the degree-$n$ map 
	\[
	    \tau_n : \BA^1_{h^{1/n}} \to \BA^1_h
	\]
	defined by $c \mapsto c^n$.\footnote{We will use the notations $\BA^1_{h^{1/n}}$ and $\BA^1_h$ to distinguish between the domain and target of the map $\tau_n$. Of course, both are abstractly isomorphic to $\BA^1$.} Consider the standard $\gm$-actions on the domain and target. The map $\tau_n$ intertwines these actions, up to the degree-$n$ group homomorphism $\gm \xrightarrow{p_n} \gm$. The goal of this subsection is to relate the categories $\QCoh^{\gm}_{\tf}(\BA^1_h)$ and $\QCoh^{\gm}_{\tf}(\BA^1_{h^{1/n}})$.
	
	\begin{lem} \label{lem-sub} 
		Consider the commutative diagram of stacks 
		\begin{cd}
			\pt / \gm \ar[d] \ar[r, "p_n"] & \pt/\gm \ar[d] \\
			\BA^1_{h^{1/n}} / \gm \ar[r, "\tau_n"] & \BA^1_h / \gm
		\end{cd}
		where the vertical maps go to $0 \in \BA^1$. The induced diagram of categories
		\begin{cd}
			\QCoh^{\gm}_{\tf}(\BA^1_h) \ar[r, "\tau_n^*"] \ar[d] & \QCoh_{\tf}^{\gm}(\BA^1_{h^{1/n}}) \ar[d] \\
			\QCoh^{\gm}(\pt) \ar[r, "p_n^*"] & \QCoh^{\gm}(\pt)
		\end{cd}
		is Cartesian. Equivalently, $\tau_n^*$ is fully faithful, and its essential image is the subcategory 
		\[
			\QCoh^{\gm}_{\tf, (n)}(\BA^1_{h^{1/n}}) \subset \QCoh_{\tf}^{\gm}(\BA^1_{h^{1/n}})
		\]
		consisting of $\gm$-equivariant sheaves $\mc{F}$ for which the weights of the action $\gm \acts \mc{F}|_0$ are all multiples of $n$.  
	\end{lem}
	\begin{proof}
		We apply Proposition~\ref{ar} to the above diagram of categories. Under this reformulation, the map $\tau_n^*$ sends a filtered vector space $(V, F)$ to $(V, F')$ where $F'$ is the filtration defined by $(F')^{\ge m} V = F^{\ge nm} V$. The map $p_n^*$ sends a graded vector space $\oplus_i V_i$ to the same vector space but with the degrees multiplied by $n$. The vertical maps send a filtered vector space to its associated graded. 
		
		Now the claim follows because a filtered vector space is of the form $(V, F')$ if and only if its associated graded has nonzero components only in degrees which are multiples of $n$. 
	\end{proof}
	
	\begin{rmk}
		The restriction to torsion-free sheaves is necessary. An example of a torsion sheaf on $\BA^1_{h^{1/n}}$ which satisfies the condition that the weights of the $\gm$-action on the zero fiber are multiples of $n$ but which does not lie in the essential image of $\tau_n^*$ is a skyscraper sheaf at $0 \in \BA^1_{h^{1/n}}$ with $\gm$-weight 0. 
	\end{rmk}
	
	Since the functors $\tau_n^*$ and $p_n^*$, one can immediately bootstrap Lemma~\ref{lem-sub} from a statement about sheaves on $\BA^1_h / \gm$ to a statement about stacks (and sheaves on stacks) over $\BA^1_h / \gm$, by the reasoning in Remark~\ref{rmk-monoidal}. This yields the following: 
	
	\begin{prop} \label{cor-cover} 
		Let $\mc{Y}$ be a stack with $\gm$-action, and let $\mc{Y} \to \BA^1_h$ be a $\gm$-equivariant map. Let $\tau_n^*\mc{Y}$ be the base change along $\tau_n : \BA^1_{h^{1/n}} \to \BA^1_h$. Then the following diagram is Cartesian and the horizontal functors are fully faithful: 
		\begin{cd}
			\QCoh^{\gm}_{\tf}(\mc{Y}) \ar[r, "\tau_n^*"] \ar[d] & \QCoh^{\gm}_{\tf}(\tau_n^*\mc{Y}) \ar[d] \\
			\QCoh^{\gm}_{\tf}(\mc{Y}|_0) \ar[r, "p_n^*"] & \QCoh^{\gm}_{\tf}((\tau_n^*\mc{Y})|_0)
		\end{cd}
	\end{prop}
	Here the subscript `$\tf$' refers to sheaves which are flat over the base $\BA^1$, not necessarily flat over $\mc{Y}$. The diagram of categories is induced from an analogous diagram of stacks, and the symbols $\mc{Y}|_0$ and $(\tau_n^*\mc{Y})|_0$ refer to the fibers over $0 \in \BA^1_h$ and $0 \in \BA^1_{h^{1/n}}$, respectively. Note that $\mc{Y}|_0 \simeq (\tau_n^*\mc{Y})|_0$, but the natural $\gm$-actions on them differ by $p_n$, which is why we denote the bottom horizontal map by $p_n^*$. 
	
	Applying Proposition \ref{cor-cover} to $\mc{Y} = \BA^1 / G^s$ yields the following: 
	
	\begin{cor}\label{cor-cover-apply}
		We have an equivalence 
		\[
		\QCoh^{\gm}_{\tf}(\BA^1_h / G^s) \xrightarrow{\tau_n^*} \QCoh^{\gm}_{\tf, (n)}(\BA^1_{h^{1/n}} / \tau_n^*G^s), 
		\]
		where the right hand side is the full subcategory of $\QCoh^{\gm}_{\tf}(\BA^1_{h^{1/n}} / \tau_n^*G^s)$ consisting of sheaves whose pullback to the zero fiber $\pt / G^{s(0)}$ lies in the essential image of 
		\[
		\QCoh^{\gm}(\pt / G^{s(0)}) \xrightarrow{p_n^*} \QCoh^{\gm}(\pt / G^{s(0)}). 
		\]
	\end{cor}
	
	Again, note that the two $\gm$-actions on $\pt / G^{s(0)}$ differ by $p_n$. 
	
	\subsection{A \texorpdfstring{$\gm$}{gm}-action on \texorpdfstring{$G^s$}{Gs}} \label{slicing-gm}
	
	From this point onward, we use the cocharacter $\gamma : \gm \to G$ included in the fastening datum $(\gamma, \ell)$. This cocharacter will allow us to resolve the issue mentioned at the end of~\ref{slice-line}. 
	
	The conditions of Definition~\ref{def:fasteningdatum} imply that there is a unique \emph{nontrivial} action $\gm \acts \BA^1$ for which the map $\BA^1 \xrightarrow{\ell} X$ is $\gm$-equivariant, where the action $\gm \acts X$ occurs via $\gamma$. Indeed, $\ell$ induces an isomorphism between $\BA^1$ and the locally closed subscheme $\ell(\BA^1) \subset X$, and the latter is invariant under the action of $\gamma(\gm)$. 
	
	Furthermore, Definition~\ref{def:fasteningdatum}(1) implies that this action $\gm \acts \BA^1$ fixes $0 \in \BA^1$, so this action is linear. Thus it is given by $t \cdot h = t^n h$ for some nonzero integer $n$. Replacing $\gamma$ with $\gamma^{-1}$ if necessary, we can assume that $n > 0$. We specialize the general constructions of~\ref{descent} to the case of this particular $n$. 
	
	Define a $\gm$-action on $G \times \BA^1_{h^{1/n}}$ by the formula
	\[
		(g,x) \cdot t= (\gamma(t)g\gamma(t)^{-1}, t x ) \qquad t \in \gm, g \in G, x \in \BA^1_{h^{1/n}}. 
	\]
	The compatibility between $\gamma$ and $\ell$ implies that the group subscheme $\tau_n^*G^s \subset G \times \BA^1_{h^{1/n}}$ is invariant under this action. Hence, we obtain an action
	\begin{equation}\label{eq-act} 
	\tau_n^*G^s \lacts \gm. 
	\end{equation}

	\begin{lem}\label{lem-crux}
		The action $\BA^1_{h^{1/n}} / \tau_n^*G^s \lacts \gm$ induced from (\ref{eq-act}) agrees with the pullback of the action $\BA^1_h / G^s \lacts \gm$ along the map $\tau_n$. 
	\end{lem}
	\begin{proof}
		Let us start with the action $\BA^1_{h^{1/n}} / \tau_n^*G^s \lacts \gm$ obtained by pullback along $\tau_n$. We will define a $\gm$-equivariant map 
		\[
		\BA^1_{h^{1/n}} \xrightarrow{\phi} \BA^1_{h^{1/n}} / \tau_n^*G^s
		\]
		of stacks over $\BA^1_{h^{1/n}}$ such that the resulting $\gm$-action on the relative stabilizer scheme obtained by the fibered product 
		\begin{cd}
			\tau_n^*G^s \ar[r] \ar[d] & \BA^1_{h^{1/n}} \ar[d, "\phi"] \\
			\BA^1_{h^{1/n}} \ar[r, "\phi"] & \BA^1_{h^{1/n}} / \tau_n^* G^s
		\end{cd}
		coincides with the action of~(\ref{eq-act}). This proves the lemma. 
		
		Recall from~\ref{slice-line} that the action $\BA^1_h / G^s \lacts \gm$ was defined using the isomorphism 
		\[ 
		G \back \I^\circ_{Z/X} \simeq \BA^1_h / G^s 
		\]
		and the action $\I^\circ_{Z/X} \lacts \gm$ which scales the fibers of the line bundle. We display these actions in the following diagram: 
		\begin{cd}
			G \ar[draw=none, shift right = 1]{r}{\acts} \ar[d, dash, shift left = 0.5] \ar[d, dash, shift right = 0.5] & G \times \BA^1_h \ar[d, swap, "a"] \\
			G \ar[draw=none, shift right = 1]{r}{\acts} & \mc{I}^\circ_{Z/X} \ar[draw=none, shift right = 1]{r}{\lacts} & \gm
		\end{cd}
		Consider the base change along $\tau_n$: 
		\begin{cd}
			G \ar[draw=none, shift right = 1]{r}{\acts} \ar[d, dash, shift left = 0.5] \ar[d, dash, shift right = 0.5] & G \times \BA^1_{h^{1/n}} \ar[d, swap, "a'"] &  \\
			G \ar[draw=none, shift right = 1]{r}{\acts} & \mc{I}^\circ_{Z/X} \underset{\BA^1_h}{\times} \BA^1_{h^{1/n}} \ar[draw=none, shift right = 1]{r}{\lacts} & \gm
		\end{cd}
		Here, $\gm$ acts on $\I^\circ_{Z/X}$ via the $n$-th power of the old action, and it acts on $\BA^1_{h^{1/n}}$ via the standard action. 
		
		We define an action $(G \times \BA^1_{h^{1/n}}) \lacts \gm$ via the following formula: 
		\begin{equation}\label{eq-act2} 
		(g, c) \cdot t = (g \cdot \gamma^{-1}(t), tc). 
		\end{equation}
		We claim that this action ensures that $a'$ is $\gm$-equivariant. Since the $\gm$-actions take place on irreducible varieties, it suffices to check this agreement on a dense open subset of the domain. Therefore, we can restrict to the locus $G \times (\BA^1_{h^{1/n}} \setminus \{0\})$ and use the isomorphism $\I^\circ_{Z/X}|_{\BA^1_h \setminus \{0\}} \simeq U \times (\BA^1_h \setminus \{0\})$ which was noted in Remark~\ref{rmk-I-geo}. With this identification in hand, the desired commutativity follows from a direct calculation: 
		\begin{cd}
			(g, c) \ar[rr, mapsto, "\act_t"] \ar[d, mapsto, "a'"] & & (g \cdot \gamma^{-1}(t), tc) \ar[d, mapsto, "a'"] \\
			\big((g \cdot \ell(c^n), c^n), c\big) \ar[r, mapsto, "\act_t"] & \big((g \cdot \ell(c^n), t^nc^n), tc\big) \ar[r, shift left = 0.5, dash] \ar[r, shift right = 0.1, dash] & \big((g \cdot \gamma^{-1}(t)\cdot \ell(t^nc^n), t^nc^n), tc\big)
		\end{cd}
		for $g \in G,\ c \in \BA^1_{h^{1/n}} \setminus \{0\}, \ t \in \gm$. 
		Thus, we get a diagram 
		\begin{cd}
			G \ar[draw=none, shift right = 1]{r}{\acts} \ar[d, dash, shift left = 0.5] \ar[d, dash, shift right = 0.5]& G \times \BA^1_{h^{1/n}} \ar[d, swap, "a'"] \ar[draw=none, shift right = 1]{r}{\lacts}  & \gm \ar[d, dash, shift left = 0.5] \ar[d, dash, shift right = 0.5] \\
			G \ar[draw=none, shift right = 1]{r}{\acts} & \mc{I}^\circ_{Z/X} \underset{\BA^1_h}{\times} \BA^1_{h^{1/n}} \ar[draw=none, shift right = 1]{r}{\lacts} & \gm
		\end{cd}
		Upon passing to $G$-quotients, the map $a'$ becomes the desired map $\phi$. 
		
		\begin{rmk*}
			Let us summarize the previous construction in more conceptual language. The map $\phi$ corresponds to the choice of a $\gm$-equivariant $\tau_n^*G^s$-torsor on $\BA^1_{h^{1/n}}$. From this point of view, the original map $\BA^1_h \to \BA^1_h / G^s$ corresponds to the (trivial) $G^s$-torsor on $\BA^1_h$ obtained by taking the pullback of $a$ along $s$. Unfortunately, this $G^s$-torsor cannot be made $\gm$-equivariant over $\BA^1_h$. However, once we pull everything back along $\tau_n$, the resulting (trivial) $\tau_n^*G^s$-torsor does admit a structure of $\gm$-equivariance, which yields $\phi$. 
		\end{rmk*}
		
		It remains to check that the action $\tau_n^*G^s \lacts \gm$ coincides with the action of~(\ref{eq-act}). We have a fibered diagram\footnote{This diagram is the $\tau_n$-pullback of part of the (augmented) simplicial object in Lemma~\ref{lem-gs}(iii) which, to use the terminology in that lemma, encodes the right action of $\mc{G}^s$ on $\mc{G}$.}  
		\begin{cd}
			G \times \tau_n^*G^s \ar[r, "\id_G \times \pi'"] \ar[d, swap, "{(\on{mult}, \pi')}"] & G \times \BA^1_{h^{1/n}} \ar[d, "a'"] \\
			G \times \BA^1_{h^{1/n}} \ar[r, "a'"] & \I^\circ_{Z/X} \underset{\BA^1_h}{\times} \BA^1_{h^{1/n}}
		\end{cd}
		Here $\pi'$ is the projection to the base $\BA^1_{h^{1/n}}$ and $\on{mult}$ is the group multiplication. The $\gm$-actions on the three lower factors collectively induce a $\gm$-action on the upper-left factor, which one checks is given by the formula 
		\[
		\big(g, (g', c)\big) \cdot t = \big(g\, \gamma^{-1}(t), (\gamma(t)\, g' \, \gamma^{-1}(t), tc)\big)
		\]
		for $g \in G,\ c \in \BA^1_{h^{1/n}},\ g' \in \tau_n^*G^s|_{c} \simeq G^{s(c^n)},\ t \in \gm$. On the upper-left factor, passing to the $G$-quotient amounts to forgetting the first coordinate, and the resulting $\gm$-action is exactly that of~(\ref{eq-act}). 
	\end{proof}
	
	\subsection{Applying the criterion for descent}
	
	Let $\phi : \BA^1_{h^{1/n}} \to \BA^1_{h^{1/n}} / \tau_n^* G^s$ be the $\gm$-equivariant map constructed in the proof of Lemma~\ref{lem-crux}. Given a quasicoherent sheaf $\mc{F} \in \QCoh^{\gm}(\BA^1_{h^{1/n}} / \tau_n^*G^s)$, the pullback $\phi^*\mc{F}$ is a sheaf on $\BA^1_{h^{1/n}}$ which is equivariant with respect to $\gm$ and the group scheme $\tau_n^*G^s$. These two actions are intertwined (in an obvious manner) with the action $\tau_n^*G^s \lacts \gm$ defined in~\ref{slicing-gm}(\ref{eq-act}).\footnote{One could formulate this as an action of the semidirect product $\tau_n^*G^s \rtimes \gm$ on the sheaf $\phi^*\mc{F}$ on $\BA^1_{h^{1/n}}$.} 
	
	We can now characterize the subcategory $\QCoh^{\gm}_{\tf, (n)} (\BA^1_{h^{1/n}} / \tau_n^*G^s)$ in concrete terms: 
	
	\begin{cor}\label{cor-n} 
		\begin{enumerate}[label=(\roman*)]
			\item[] 
			\item Let $G^{s(0)} \rtimes \gm$ be defined using the $\gamma$-conjugation action of \ref{slicing-gm}(\ref{eq-act}). The image of the group homomorphism $G^{s(0)} \rtimes \gm \xrightarrow{(\iota, \gamma)} G$ is $G^{\ell(0)}$. 
			\item The subcategory $\QCoh^{\gm}_{\tf, (n)}(\BA^1_{h^{1/n}} / \tau_n^*G^s)$ defined in Corollary~\ref{cor-cover-apply} consists of quasicoherent sheaves $\mc{F}$ satisfying the following condition: 
			\begin{itemize}
				\item The action $G^{s(0)} \rtimes \gm \acts \mc{F}|_0$ factors through the map $(\iota, \gamma)$ in (i). 
			\end{itemize}
		\end{enumerate}
	\end{cor}
	
	\subsubsection{Proof of (i)} \label{groupss}
	Consider the $G$-equivariant projection map $\mc{N}^\circ_{Z/X} \to Z$. Since $\gamma(\gm)$ acts transitively on the fiber over $\ell(0) \in Z$, and since this fiber contains $s(0) \in \mc{N}^\circ_{Z/X}$, we have $G^{\ell(0)} = \gamma(\gm) \cdot G^{s(0)}$ as subgroups of $G$. This proves (i). 
	
	For expository purposes, let us record a few additional relations between these groups. We have a diagram of exact sequences 
	\begin{equation}\label{groups} 
	\begin{tikzcd}
	& \mu_n \ar[r, dash, shift left = 0.5] \ar[r, dash, shift right = 0.5] \ar[d, hookrightarrow, "{(\gamma \circ \on{inv} , \iota)}"] & \mu_n \ar[d, hookrightarrow] \\
	G^{s(0)} \ar[d, dash, shift right = 0.5] \ar[d, dash, shift left = 0.5] \ar[r, hookrightarrow] & G^{s(0)} \rtimes \gm \ar[r, twoheadrightarrow] \ar[d, twoheadrightarrow, "{(\iota, \gamma)}"] & \gm \ar[d, twoheadrightarrow, "n"] \\
	G^{s(0)} \ar[r, hookrightarrow] & G^{\ell(0)} \ar[r, twoheadrightarrow, "\act"] & \gm
	\end{tikzcd}
	\end{equation}
	Here, the map $\act$ encodes the action of $G^{\ell(0)}$ on the normal line to $z \in Z \subset X$ (i.e.\ the fiber mentioned in the previous paragraph), and the map `inv' is group inversion. 
	
	\subsubsection{Proof of (ii)}
	
	To unpack the definition of $\QCoh^{\gm}_{\tf, (n)}(\BA^1_{h^{1/n}} / \tau_n^*G^s)$, we need to describe the map of stacks 
	\[
		(\pt / G^{s(0)}) / \gm^{(n)} \xrightarrow{p_n} (\pt / G^{s(0)}) / \gm
	\]
	which appears in Corollary~\ref{cor-cover-apply}.\footnote{Here, and in what follows, we use the notation $\gm^{(n)}$ to indicate a copy of $\gm$ which acts via $n$ times the standard action. This is to avoid collision of notation for nonisomorphic stacks.} As noted there, the map is the identity on the underlying stack $\pt / G^{s(0)}$, but the two $\gm$-actions differ by $p_n$, the $n$-th power group homomorphism $\gm \to \gm$. In view of Lemma~\ref{lem-gs} and Remark~\ref{rmk-orbits}, this map is given by 
	\[
		G\ \backslash \ ( \mc{N}^\circ_{Z/X} ) \ / \ \gm^{(n)} \to G \ \backslash \ ( \mc{N}^\circ_{Z/X} ) \ / \ \gm.
	\]
	The $\gm$-action on the right hand side is the standard one which scales the fibers of the normal bundle, while the $\gm$-action on the left hand side is $n$ times the former. Choosing the basepoint $s(0) \in \mc{N}^\circ_{Z/X}$, this becomes the map 
	\[
	\pt / \stab_{G \times \gm^{(n)}}(s(0)) \to \pt / \stab_{G \times \gm}(s(0))
	\]
	which is induced by the map of groups $\psi_1$ in this Cartesian diagram: 
	\begin{cd}
		\stab_{G \times \gm^{(n)}}(s(0)) \ar[r, "\psi_1"] \ar[d, hookrightarrow] & \stab_{G \times \gm}(s(0)) \ar[r, "\sim"] \ar[d, hookrightarrow] & G^{\ell(0)}  \\
		G \times \gm^{(n)} \ar[r, "\id_G \times p_n"] & G \times \gm
	\end{cd}
	This diagram contains the additional observation that $\stab_{G \times \gm}(s(0)) \simeq G^{\ell(0)}$, which follows because $Z \simeq \mc{N}^\circ_{Z/X} / \gm$ using the standard $\gm$-action, and under this quotient map $s(0) \in \mc{N}^\circ_{Z/X}$ projects to $\ell(0) \in Z$. 
	
	We extend this by another Cartesian square
	\begin{cd}
		G^{s(0)} \rtimes \gm^{(n)} \ar{r}{\psi_2'}[swap]{\sim} \ar[d, hookrightarrow, "\iota \times \id_{\gm}"] & \stab_{G \times \gm^{(n)}}(s(0)) \ar[r, "\psi_1"] \ar[d, hookrightarrow] & \stab_{G \times \gm}(s(0)) \ar[r, "\sim"] \ar[d, hookrightarrow] & G^{\ell(0)} \\
		G \rtimes \gm^{(n)} \ar{r}{\psi_2}[swap]{\sim} & G \times \gm^{(n)} \ar[r, "\id_G \times p_n"] & G \times \gm
	\end{cd}
	Here, the map $\psi_2$ sends $(g, t) \mapsto (g\, \gamma^{-1}(t), t)$, which is a group isomorphism by definition of the action~\ref{slicing-gm}(\ref{eq-act}). The pullback of $\stab_{G \times \gm^{(n)}}(s(0))$ under $\psi_2$ is indeed $G^{s(0)} \rtimes \gm^{(n)}$, because the formula for $\psi_2$ shows that the whole subgroup $\gm^{(n)} \subset G \rtimes \gm^{(n)}$ lies in the stabilizer. And it is straightforward to check that the upper horizontal composition $G^{s(0)} \rtimes \gm^{(n)} \sra G^{\ell(0)}$ coincides with the map $(\iota, \gamma)$ from (i), by comparing the formula for $\psi_2$ with the formula~\ref{slicing-gm}(\ref{eq-act2}). 
	
	Thus, the map of stacks written at the beginning of this subsubsection identifies (via the isomorphisms of the previous paragraph) with the map
	\[
		\pt / (G^{s(0)} \rtimes \gm^{(n)}) \to \pt / G^{\ell(0)}
	\]
	which is induced by the group homomorphism $(\iota, \gamma)$ from (i). Tracing through the construction of Lemma~\ref{lem-crux}, we find that the action $G^{s(0)} \rtimes \gm \acts \mc{F}|_0$ mentioned in (ii) coincides with the action obtained by pulling back along 
	\[
		\pt / (G^{s(0)} \rtimes \gm^{(n)}) \simeq (\pt / G^{s(0)}) / \gm^{(n)} \hra \BA^1_{h^{1/n}} / \tau_n^* G^s. 
	\]
	This proves (ii). \hfill $\square$
	
	\subsection{The main result} 
	
	In this subsection, we assemble the previous constructions to obtain an algebraic description of the category $\QCoh^G_{\tf}(X)$. 
	\begin{defn}\label{claim1}
		Let $A$ be the filtered Hopf algebra corresponding to the group scheme $\tau_n^*G^s \to \BA^1_{h^{1/n}}$ under Proposition~\ref{ar} and Remark~\ref{rmk-monoidal}. This map is flat by Lemma~\ref{lem-gs}(i), and~\ref{slicing-gm}(\ref{eq-act}) gives a $\gm$-action on $\tau_n^*G^s$, so the Artin--Rees construction applies. As is always the case with Artin--Rees, we have 
		\e{
			\Spec A &\simeq \big((\tau_n^*G^s)|_{\BA^1_{h^{1/n}} \setminus \{0\}}\big) / \gm \\
			\Spec \gr A &\simeq \tau_n^*G^s|_0 \simeq G^{s(0)}. 
		} 
		In the second identification, the $\gm$-action on the left hand side given by the grading corresponds to the usual $\gm$-action on the right hand side (via $\gamma$). 
	\end{defn}
	\begin{theorem}\label{thm-main}
		Let $\mc{C}_{(n)}(X)$ be the category whose objects are triples $(V, F, \alpha)$ as follows: 
		\begin{itemize}
			\item $(V, F)$ is an exhaustively filtered vector space, and $\alpha : V \to V \otimes A$ is a coaction map which makes $(V, F)$ a filtered comodule for $A$. 
			\item We furthermore require that the associated graded action 
			\[
			\gm \ltimes G^{s(0)} \acts \gr^F V
			\]
			factors through the quotient $\gm \ltimes G^{s(0)} \sra G^{\ell(0)}$ from Corollary~\ref{cor-n}(i).
		\end{itemize}
		The morphisms in this category are those of filtered $A$-comodules. Then there are equivalences of categories 
		\[
		\QCoh^G_{\tf}(X) \simeq \QCoh^{\gm}_{\tf, (n)}(\BA^1 / \tau_n^*G^s) \simeq  \mathcal{C}_{(n)}(X). 
		\]
	\end{theorem}
	\begin{proof}
		The first equivalence follows from Corollary~\ref{cor-cover-apply}. The second equivalence follows from the definition of $A$ and the last sentence of Remark~\ref{rmk-monoidal}. The reformulation of the $(-)_{(n)}$ condition given in the theorem statement follows from Corollary~\ref{cor-n}(ii). 
	\end{proof}
	
	\begin{rmk}
		If one desires to remove the torsion-free requirement, then in Proposition~\ref{ar} the notion of `filtered vector space' must be replaced with that of `graded vector space equipped with a degree 1 endomorphism' and the criterion for descent given in Lemma~\ref{lem-sub} must be reformulated accordingly. 
	\end{rmk}

	\begin{rmk} \label{rmk-concrete} 
		Fixing the nonzero point $1 \in \BA^1_{h^{1/n}}$ allows one to view these constructions more concretely (albeit less canonically): 
		\begin{enumerate}[label=(\roman*)]
			\item We get identifications
			\e{
				\Spec A &\simeq G^{s(1)}\\
				\Spec \gr A &\simeq G^{s(0)}, 
			} 
			where the identifications in the second line are $\gm$-equivariant. The filtration $F^{\ge m}$ induced on the ring of functions $\oh_{G^{s(1)}}$ is given as follows: for a function $f$, we have $f \in F^{\ge m}$ if and only if the limit 
			\[
			\lim_{t \to 0} \big(t^{-m} \act_{t^{-1}}^*f\big)
			\]
			exists as a function on $G^{s(0)}$. Here $\act_{t^{-1}}$ is the isomorphism $G^{s(t^n)} \to G^{s(1)}$ given by the $\gm$-action on $G^s$. 
			\item Given an object $\mc{F} \in \QCoh^{G}_{\tf}(X)$ corresponding to $(V, F, \alpha)$ via Theorem~\ref{thm-main}, the fiber $\mc{F}|_{\ell(1)}$ identifies with $V$, and the action of $G^{\ell(1)}$ on this fiber is given by the coaction $\alpha$ (forgetting the filtrations). The fiber $\mc{F}|_{\ell(0)}$ is given by $\gr V$. 
			\item The action of $G^{\ell(0)}$ on the fiber $\mc{F}|_{\ell(0)}$ is determined as follows. The action of $G^{s(0)}$ on $\mc{F}|_{\ell(0)}$ is given by $\gr \alpha$. Since this coaction map is graded, the action of $G^{s(0)}$ extends to an action by $G^{s(0)} \rtimes \gm$. The second bullet point in Theorem~\ref{thm-main} says that the action descends to one by the quotient group $G^{\ell(0)}$. 
		\end{enumerate} 
	\end{rmk}

    \section{Specialization maps} \label{sec:specialization}

	Retain the notations and assumptions of Section~\ref{sec:main}. This means that $X = U \cup Z$ is a smooth fastened chain equipped with fastening datum $(\gamma, \ell)$, which determines the group schemes $G^s$ and $\tau_n^*G^s$, the latter being $\gm$-equivariant over $\BA^1$. Theorem~\ref{thm-main} gives us an equivalence between $\QCoh^G_{\tf}(X)$ and a category of filtered co-modules for the filtered Hopf algebra $A = \oh_{G^{s(1)}}$. In this section, we show that some objects in $\QCoh^G_{\tf}(X)$ can be presented more economically, without referring explicitly to $A$. 
	
	In the first half of this section, we give nicer classifications of line bundles (Proposition~\ref{prop:linebundles}) and sheaves whose restrictions to $U$ are $G$-equivariant local systems (Corollary~\ref{cor-locsys}). These classifications are `nicer' in the sense that the filtered Hopf algebra $A$ does not explicitly appear; instead, in each case $\tau_n^*G^s$ acts through a certain quotient 
	\[
	    q : \tau_n^*G^s \to H
	\]
	(see~\ref{s:factor}), where $H \to \BA^1$ is a $\gm$-equivariant group scheme with the property that the action $\gm \acts H|_0$ is trivial. This implies that the specialization to the zero fiber for $H$ is encoded in a \emph{bona fide} map of algebraic groups $H|_0 \to H|_1$. 
	
	In the second half of this section, we investigate the largest subcategory of $\Vect^G(X)$ for which this works. We show that there is a largest possible quotient of $\tau_n^*G^s$ which admits a `specialization map' description (Corollary~\ref{cor-tame-univ}) and we state the corresponding classification of vector bundles in Corollary~\ref{cor-tame-classify}. Lastly, in~\ref{apply}, we apply this idea to the problem of classifying sheaves whose restrictions to $U$ are \emph{twisted} $G$-equivariant local systems, which is the case of interest for our application (Section~\ref{sec:application}). 
	
	\noindent \emph{Notations.} For the remainder of this section, we adopt the more concise notation $\tilde{G}^s = \tau_n^*G^s$. Recall from Remark~\ref{rmk-orbits} that $G^{s(0)} \simeq G^s|_0$ and $G^{s(1)} \simeq G^s|_1$. As a matter of convention, any scheme written as $S \times \BA^1$ (for some $S$) will be equipped with the $\gm$-action given by the standard action on the second factor. For an arbitrary map $H \to \BA^1$, we let $H_0$ and $H_1$ denote the fibers over $0, 1 \in \BA^1$, respectively. For a map $q$ of schemes over $\BA^1$, we let $q_0$ and $q_1$ denote the corresponding fibers of this map. 
	
	\subsection{Line bundles} \label{s:line}
	In this subsection, we give a very simple description of the category of $G$-equivariant line bundles (Proposition~\ref{prop:linebundles}(ii)), which will be applied in~\ref{ssec:case-line}. Along the way, we prove a technical result which may be of independent interest: the specialization of a character of $G^{s(1)}$ to the zero fiber $G^{s(0)}$ is always well-defined (Proposition~\ref{prop:linebundles}(i)). This result is applied again in~\ref{apply}. 
	
	Let $\Pic^G(X)$ denote the category of $G$-equivariant line bundles on $X$. For any character $\chi : G^{s(1)} \to \gm$, let $\Pic^G_{\chi}(X)$ be the full subcategory consisting of line bundles $\mc{L}$ for which the action $G^{s(1)} \acts \mc{L}|_{s(1)}$ factors through $\chi$. We have a decomposition
	\[
	\Pic^G(X) \simeq \bigsqcup_{\chi} \Pic^G_{\chi}(X).
	\]
	
	\begin{prop}\label{prop:linebundles}
		Fix a character $\chi : G^{s(1)} \to \gm$. 
		\begin{enumerate}
			\item[(i)] We have $\chi \in F^{\ge 0} A$. Concretely, this means that $\lim_{t \to 0} \on{act}_{t^{-1}}^* \chi$ exists as a character of $G^{s(0)}$. 
		\end{enumerate} 
		Let $q : \tilde{G}^s \to \gm \times \BA^1$ be the $\gm$-equivariant map of group schemes given by (i). The map $q$ is uniquely determined by the requirement that $q_1 : G^{s(1)} \to \gm$ equals $\chi$. 
		\begin{enumerate}
			\item[(ii)] We have a canonical equivalence of groupoids 
			\[
			\Pic^G_{\chi}(X) \simeq (n\BZ + m) \times (\pt / \BC^\times). 
			\]
			The residue class $m \bmod{n}$ is determined by the following property: if $\mu_n \simeq G^{s(0)} \cap \gamma(\gm)$ (see~\ref{groupss}), then the restriction of the character 
			\begin{cd}[column sep = 0.8in]
				\mu_n \ar[r, hookrightarrow, "\on{inclusion}"] & G^{s(0)} \ar[r, "\lim_{t \to 0} \mathrm{Ad}_{\gamma(t^{-1})}^* \chi"] & \gm
			\end{cd}
			is given by $\zeta \mapsto \zeta^m$. 
		\end{enumerate}
	\end{prop}
	
	\subsubsection{Proof that (i) implies (ii) in Proposition~\ref{prop:linebundles}}  \label{prop:linebundles-1}
	
	The first equivalence in Theorem~\ref{thm-main} implies that an object $\mc{L} \in \Pic^G(X)$ is equivalent to a pair $(\varphi, \mc{L}')$ defined as follows: 
	\begin{itemize}
		\item $q : \tilde{G}^s \to \gm \times \BA^1$ is a $\gm$-equivariant homomorphism of group schemes. 
		\item $\mc{L}'$ is a $\gm$-equivariant line bundle on $\BA^1$.
		\item We require that the restriction of $\gm \acts \mc{L}'|_{0}$ to $\mu_n$ is given by the group homomorphism $\mu_n \xrightarrow{\gamma} G^{s(0)} \xrightarrow{q_0} \gm$. 
	\end{itemize}
	This is because, for any $\gm$-equivariant line bundle $\mc{L}'$ on $\BA^1$, the group scheme $\curAut_{\BA^1}(\mc{L}')$ is $\gm$-equivariantly isomorphic to $\gm \times \BA^1$. 
	
	Furthermore, $\mc{L} \in \Pic^G_{\chi}(X)$ if and only if $q_1 = \chi$. Thus $q$ is determined, and the remaining datum $\mc{L}'$ amounts to the choice of a one-dimensional vector space (which plays the role of $\mc{L}'|_{1}$) and an integer (which determines the filtration). The third bullet point above says that this integer lies in $n\BZ + m$, because $q_1 = \chi$ implies that $q_0 = \lim_{t \to 0} \on{act}_{t^{-1}}^* \chi$. \hfill $\qed$
	
	The rest of this subsection is devoted to the proof of Proposition~\ref{prop:linebundles}(i). 
	
	\begin{lem}\label{lem-dim-1}
		Let $\pi : H \to \BA^1$ be a flat $\gm$-equivariant group scheme of relative dimension 1. If there is an isomorphism $\varphi: H_1 \simeq \gm$, then the following are true: 
		\begin{enumerate}[label=(\roman*)]
			\item The $\gm$-action on $H_0$ is (weakly) expanding as $t \to 0$.\footnote{This means that $\oh_{H_0}$ lives in nonnegative degrees with respect to the $\gm$-action. We omit the word `weakly' from now on, because we will not use the `strong' notion anywhere.}
			\item There exists a (unique) $\gm$-equivariant map $q : H \to \gm \times \BA^1$ of group schemes with the property that $q_1 : H_1 \to \gm$ equals $\varphi$. 
		\end{enumerate} 
	\end{lem}
	\begin{proof}
		(i). Suppose for sake of contradiction that the action $\gm \acts H_0$ is not expanding. This implies that all weights of the $\gm$-action on the (two-dimensional) tangent space $\mc{T}_{1_{H_0}}H$ are positive (here $1_{H_0} \in H_0$ is the identity element).  (It also implies that the identity component of $H_0$ is isomorphic to $\ga$, but we will not use this explicitly.) 
		
		Applying the Luna Slice Theorem to the $\gm$-orbit $\{1_{H_0}\} \subset H$ gives us a $\gm$-invariant affine open neighborhood $W \subset H$ containing $1_{H_0}$, along with a $\gm$-equivariant \'etale map 
		\[
			\psi : W \to \mc{T}_{1_{H_0}} H
		\]
		which induces an isomorphism on tangent spaces at $1_{H_0} \in W$. (The last phrase uses that $\pi$ is smooth, which is true because $\pi$ is flat, finitely presented, and has smooth fibers.) The $\gm$-equivariance of $\psi$, along with the previous paragraph, implies that $\psi$ is surjective. This in turn implies that $H_1 \cap W \simeq \BA^1$, contradicting the assumption that $H_1 \simeq \gm$. 
		
		(ii). Point (i) says that $\oh_{H_0} \simeq \gr \oh_{H_1}$ lives in nonnegative degrees. In view of Remark~\ref{rmk-concrete}(i), this implies that any regular function on $H_1$ extends to a $\gm$-invariant function on $H$. Applying this to the character $\varphi$ yields the desired result.
	\end{proof}
	
	One can describe all possibilities for $\pi : H \to \BA^1$ satisfying the hypotheses of this lemma. See the second paragraph of Remark~\ref{rmk-nonflat} for more details. 
	
	\subsubsection{Proof of Proposition~\ref{prop:linebundles}(i)} 
	
	Consider the codimension-one subgroup $\ker(\chi) \subset G^{s(1)}$, and extend it to a $\gm$-equivariant sub group scheme $K \subset \tilde{G}^s$ which is flat over $\BA^1$. (Concretely, $K$ is obtained by taking the flat limit of the $\gm$-saturation of $\ker(\chi)$. It satisfies the sub group scheme property because this is a closed condition, hence this condition is preserved under taking the flat limit.) Let $\tilde{G}^s \sra H$ be the quotient by $K$, so that $H$ is a flat $\gm$-equivariant group scheme of relative dimension $\le 1$. By construction, there is a group map $H_1 \xhookrightarrow{\on{cl.emb.}} \gm$ which expresses $H_1$ as the image of $\chi$. 
	
	Assume that $H$ has relative dimension $1$, so that $H_1 \hra \gm$ is an isomorphism. Then Lemma~\ref{lem-dim-1}(ii) gives a map $q : H \to \gm \times \BA^1$ which witnesses the fact that the limit $\lim_{t \to 0} \on{act}_{t^{-1}}^*\chi$ exists. 
	
	If $H$ has relative dimension $0$, then $\chi$ is locally constant. Now the desired limit exists because two different connected components of $G^{s(1)}$ cannot specialize to the same component of $G^{s(0)}$. Indeed, this would contradict the fact that $G^{s(0)}$ is a group scheme and hence reduced. (This case can also be handled using~\ref{s:locsys}.) \hfill $\qed$

	\subsection{Factoring through a quotient of \texorpdfstring{$\tilde{G}^s$}{tildeGs}} \label{s:factor} 
	
	Forget the notation of~\ref{s:line}. Now let $H$ be any $\gm$-equivariant group scheme over $\BA^1$, and let $q : \tilde{G}^s \to H$ be a map of $\gm$-equivariant group schemes over $\BA^1$. 
	
	\begin{defn}
		Let $\QCoh^{\gm}_{\tf, (n)}(\BA^1 / H)$ be the full subcategory of $\QCoh^{\gm}_{\tf}(\BA^1 / H)$ consisting of sheaves $\mc{F}$ for which the restriction of the action $\gm \ltimes H_0 \acts \mc{F}|_0$ along the group map 
		\[
		q_0 : \gm \ltimes G^{s(0)} \to \gm \ltimes H_0
		\]
		factors through the quotient $\gm \ltimes G^{s(0)} \to G^{\ell(0)}$. 
	\end{defn}
	
	We would like to think of $\QCoh^{\gm}_{\tf, (n)}(\BA^1 / H)$ as encoding a subcategory of $\QCoh_{\tf}^G(X)$ which `corresponds' to the quotient $q$. The next result makes this precise. In~\ref{s:subcat}, we explain how to classify the objects of this subcategory. 
	
	\begin{prop}\label{flat-tame-prop}
		We have the following: 
		\begin{enumerate}[label=(\roman*)]
			\item If $q_1 : G^{s(1)} \to H_1$ is surjective and $H$ is flat (i.e.\ torsion-free) over $\BA^1$, then the pullback functor 
			\[
			\QCoh^{\gm}_{\tf, (n)}(\BA^1 / H) \xrightarrow{q^*}  \QCoh^{\gm}_{\tf, (n)}(\BA^1 / \tilde{G}^s)
			\]
			is fully faithful. 
			\item If furthermore $q_0 : G^{s(0)} \to H_0$ is surjective, then the image of this functor is the subcategory of sheaves $\mc{F}$ for which the action $G^{s(1)} \acts \mc{F}|_1$ factors through $q_1$. 
		\end{enumerate}
	\end{prop}
	\begin{proof}
		(i). Since $H$ is flat, an object in the left hand side is given by $(V, F, \beta)$ where $(V, F)$ is a filtered vector space and $\beta : V \to V \otimes \oh_{H_1}$ is a filtered coaction. Since $q_1$ is surjective, $\oh_{H_1} \to \oh_{G^{s(1)}}$ is injective, so whether or not a filtered coaction $V \to V \otimes \oh_{G^{s(1)}}$ factors through a coaction $V \to V \otimes \oh_{H_1}$ is a property, not additional structure. 
		
		(ii). The crux of the proof is the following lemma, which the second author learned from Dori Bejleri: 
		\begin{lem*}
			Let $f : X \to Y$ be a map of schemes over a base $S$. Let $x \in X$ be a point, and let $y \in Y$ and $s \in S$ be its images. Then the following two statements are equivalent: 
			\begin{itemize}
				\item $X\to S$ is flat at $x$, and $f|_s : X|_s \to Y|_s$ is flat at $x$. 
				\item $Y \to S$ is flat at $y$, and $f$ is flat at $x$. 
			\end{itemize}
		\end{lem*} 
		Since $q_0$ is a surjective map of algebraic groups, it is flat. Since $\tilde{G}^s$ is flat over $\BA^1$, we may apply the lemma and conclude that $q$ is flat. Since flat maps are preserved by base change, we conclude that the scheme $\ker(q)$ is flat over $\BA^1$. 
		
		The image of the functor is obviously contained in the indicated subcategory. For the reverse containment, suppose that the sheaf $\mc{F} \in \QCoh^{\gm}_{\tf, (n)}(\BA^1 / \tilde{G}^s)$ lies in the indicated subcategory. The action of $\tilde{G}^s$ corresponds to a map $\tilde{G}^s \xrightarrow{\varphi} \curEnd_{\BA^1}(\mc{F})$. The closed subscheme $\ker(\varphi)|_{\BA^1 \setminus \{0\}}$ contains $\ker(q)|_{\BA^1 \setminus \{0\}}$ by the hypothesis on $\mc{F}$. The flatness proved in the previous paragraph implies that $\ker(\varphi)$ contains $\ker(q)$. Therefore the action of $\tilde{G}^s$ factors through $q$, as desired. 
	\end{proof}
	
	\begin{defn}\label{def-q-cat} 
		As a notational convenience, when the hypotheses of Proposition~\ref{flat-tame-prop}(i) are satisfied, we define $\QCoh^{G}_{\tf, q}(X)$ via the following Cartesian diagram: 
		\begin{cd}
			\QCoh^{G}_{\tf, q}(X) \ar[r, hookrightarrow] \ar{d}[rotate=90, anchor=south, swap]{\sim} & \QCoh^{G}_{\tf}(X) \ar{d}[rotate=90, anchor=south, swap]{\sim}{\text{Thm.~\ref{thm-main}}} \\
			\QCoh^{\gm}_{\tf, (n)}(\BA^1 / H) \ar[r, hookrightarrow, "q^*"] & \QCoh^{\gm}_{\tf, (n)}(\BA^1 / \tilde{G}^s)
		\end{cd}
		Note that this subcategory of $\QCoh^{G}_{\tf, (n)}(X)$ \emph{a priori} depends on the choice of fastening $(\gamma, \ell)$ as well as the map $q$. But when Proposition~\ref{flat-tame-prop}(ii) applies, it tells us that this subcategory consists exactly of those sheaves $\mc{F} \in \QCoh^G_{\tf}(X)$ for which the action $G^{s(1)} \acts \mc{F}|_{\ell(1)}$ factors through $q_1$. 
		
		This definition will appear only in Corollary~\ref{cor-tame-classify}. 
	\end{defn}
	
	\subsection{Local systems} \label{s:locsys}
	
	We apply~\ref{s:factor} to the problem of describing the subcategory 
	\[
	\QCoh_{\tf, \locsys}^G(X) \hra \QCoh_{\tf}^G(X) 
	\]
	consisting of sheaves $\mc{F}$ for which the restriction $\mc{F}_U$ is a $G$-equivariant local system. 
	
	We construct a quotient map $q : \tilde{G}^s \to H_{\locsys}$ of group schemes over $\BA^1 / \gm$ as follows. For each connected (equivalently, irreducible) component $C \subset \tilde{G}^s$, let $H_C \to \BA^1$ be the image of the projection map $C \to \BA^1$. Since $\tilde{G}^s$ is flat over $\BA^1$, each $H_C$ is either $\BA^1$ or $\BA^1 \setminus \{0\}$. We define 
	\[
	H_{\locsys} := \bigsqcup_{\substack{C \subset \tilde{G}^s \\ \text{components}}} H_C. 
	\]
	Then $H_{\locsys}$ is obviously a group scheme over $\BA^1 / \gm$, and the desired map $q$ is defined on components by the tautological map $C \to H_C$. Note that $H_{\locsys}|_1 \simeq \on{Com}(G^{s(1)})$ where $\on{Com}(-)$ means the group of components. 
	
	Since each component of $G^{s(0)}$ lies in a unique component of $\tilde{G}^s$, and the latter components are in bijection with those of $G^{s(1)}$, we obtain a canonical map of groups 
	\[
	\sigma : \on{Com}(G^{s(0)}) \to \on{Com}(G^{s(1)}). 
	\]	
	
	\begin{cor}\label{cor-locsys}
		Let $\mc{C}_{\locsys, (n)}(X)$ be the category whose objects are triples $(V, F, \rho)$ as follows: 
		\begin{itemize}
			\item $V$ is a vector space, $\rho : \on{Com}(G^{s(1)}) \acts V$ is a representation, and $F$ is an exhausting filtration on $V$ which is invariant under $\rho$. 
			\item We furthermore require that the restriction of the associated graded action 
			\[
			\gm \ltimes G^{s(0)} \to \gm \times \on{Com}(G^{s(0)}) \xrightarrow{\sigma} \gm \times \on{Com}(G^{s(1)}) \acts \gr^F V
			\]
			factors through the quotient $\gm \ltimes G^{s(0)} \sra G^{\ell(0)}$. 
		\end{itemize}
		The morphisms in this category are maps of filtered representations. Then there is an equivalence of categories
		\[
		\QCoh^G_{\tf, \locsys}(X) \simeq \mc{C}_{\locsys, (n)}(X). 
		\]		
	\end{cor}
	\begin{proof}
		Since Proposition~\ref{flat-tame-prop}(ii) applies to $q : \tilde{G}^s \to H_{\locsys}$, we conclude that the category $\QCoh^{\gm}_{\tf, (n)}(\BA^1 / H_{\locsys})$ is equivalent to the subcategory of $\QCoh^{\gm}_{\tf, (n)}(\BA^1 / \tilde{G}^s)$ consisting of sheaves $\mc{F}$ for which the action $G^{s(1)} \acts \mc{F}|_1$ factors through $G^{s(1)} \sra \on{Com}(G^{s(1)})$. But this subcategory is also equivalent to $\QCoh^G_{\tf, \locsys}(X)$ by Theorem~\ref{thm-main}. Thus, to finish the proof, it suffices to show that 
		\[
		    \QCoh^{\gm}_{\tf, (n)}(\BA^1 / H_{\locsys}) \simeq \mc{C}_{\locsys, (n)}(X).
		\]
		
		Using the Artin--Rees construction, we interpret the left hand side as consisting of triples $(V, F, \beta)$ where $(V, F)$ is as before, and $\beta : V \to V \otimes \oh_{H_{\locsys}|_1}$ is a filtered coaction. We observed above that $H_{\locsys}|_1 \simeq \on{Com}(G^{s(1)})$. If $I \subset \oh_{\on{Com}(G^{s(1)})}$ is the ideal of the subscheme $\Im(\sigma)$, then the filtration on $\oh_{H_{\locsys}|_1}$ (given by the definition of $H_{\locsys}$) corresponds to the filtration
		\begin{equation} \label{filt}
		F^{\ge m} \oh_{\on{Com}(G^{s(1)})} = \begin{cases}
		\oh_{\on{Com}(G^{s(1)})} & \text{ if } m \le 0 \\
		I & \text{ if } m > 0. 
		\end{cases}
		\end{equation}
		This filtration implies that the filtered coaction $\beta$ is equivalent to the data in the first bullet point in the definition of $\mc{C}_{\locsys, (n)}(X)$. It is easy to check that the second bullet point corresponds to the $(-)_{(n)}$ condition on the left hand side. 
	\end{proof}

	\subsection{The universal tame quotient}\label{sec:univtame}
	
	The analysis in~\ref{s:line} and \ref{s:locsys} is especially nice because in both cases we considered a quotient $q : \tilde{G}^s \to H$ (in the framework of~\ref{s:factor}) with the property that $H \to \BA^1$ is `almost' a constant group scheme. Indeed, in~\ref{s:line}, $H$ was the constant group scheme $H \simeq \gm \times \BA^1$. And in~\ref{s:locsys}, $H_{\locsys}$ is obtained from the constant group scheme $\on{Com}(G^{s(1)}) \times \BA^1$ by replacing the zero fiber by its closed subscheme $\Im(\sigma)$. 
	
	One thing these two group schemes have in common is that the $\gm$-action on their zero fibers is trivial. In what follows, we explain why this property implies that a $\gm$-equivariant scheme over $\BA^1$ is `almost constant' in a sense that will be made precise (see Lemma~\ref{tame-describe}). 
	
	\begin{defn}
	    Define the full subcategory 
	    \[
	        \QCoh_{\tf, \tame}^{\gm}(\BA^1) \subset \QCoh_{\tf}^{\gm}(\BA^1)
	    \]
	    to consist of sheaves $\mc{F} \in \QCoh^{\gm}_{\tf}(\BA^1)$ for which the action $\gm \acts \mc{F}|_0$ is trivial. We call such a sheaf \emph{tame}. Under Artin--Rees, tame sheaves correspond to filtered vector spaces $(V, F)$ with the property that $F^{\ge 0}V = V$ and $F^{\ge 1}V = F^{\infty}V$.\footnote{By definition, $F^\infty V = \bigcap_{m \in \BZ} F^{\ge m} V$ and $F^{-\infty} V = \bigcup_{m \in \BZ} F^{\ge m} V$. The filtration $F$ is exhausting if and only if $F^{-\infty} V = V$.} Such vector spaces we call \emph{tamely filtered}. 
	\end{defn}
	
	Since this is a monoidal subcategory, we obtain an analogous notion of tame affine schemes over $\BA^1 / \gm$ as in Remark~\ref{rmk-monoidal}. In the rest of the section, we study tame affine schemes in more detail. 
	
	Any (flat) affine scheme over $\BA^1 / \gm$ is given by 
	\[
	\ul{\Spec} (R, F) := \Spec \oplus_m F^{\ge m} R \xrightarrow{\pi} \Spec k[h] = \BA^1
	\]
	for some filtered ring $(R, F)$. Here $\oplus_m F^{\ge m} R$ and $k[h]$ are graded rings with $\oplus_m F^{\ge m} R$ in degree $m$ and $h$ in degree $-1$. (This convention ensures that the action $\gm \acts \BA^1_h$ is standard.) Then $\pi$ is defined by $h \mapsto 1_R \in F^{\ge -1} R$. 
	
	Now $S = \ul{\Spec}(R, F)$ is \emph{tame} if $(R, F)$ is tamely filtered, or equivalently if the action $\gm \acts S|_0$ is trivial. Note that any tame affine scheme is by definition torsion-free (equivalently flat) over $\BA^1$. 
	
	Next, we observe that tame affine schemes are `almost constant' as described in the first paragraph of this subsection. For a vector space $V$, let $F_0$ be the filtration for which $F_0^{\ge 0} = V$ and $F_0^{\ge 1}V = 0$.
	
	\begin{lem}	\label{tame-describe} 
		We have the following. 
		\begin{enumerate}[label=(\roman*)]
			\item The map of filtered rings $(F^{\ge 0} R, F_0) \to (R, F)$ yields a $\gm$-equivariant map 
			\[
			\ul{\Spec} (R, F) \xrightarrow{p} (\Spec R) \times \BA^1. 
			\]
			\item If $(R, F)$ is tamely filtered, then the base change of $p$ along $\BA^1 \setminus \{0\} \hra \BA^1$ is an isomorphism and the base change along $\{0\} \hra \BA^1$ is the closed embedding $\Spec (R/F^{\ge 1}) R \hra \Spec R$. In particular, there is a canonical map from the zero fiber of $\ul{\Spec} (R, F)$ to the general fiber. 
		\end{enumerate}
	\end{lem}
	\begin{proof}
		Point (i) is a tautology. For point (ii), if $(R, F)$ is tamely filtered, then $R = F^{\ge 0} R$, from which it follows that $p$ is an isomorphism over $\BA^1 \setminus \{0\}$. The base change of $p$ along $\{0\} \hra \BA^1$ is the induced map on associated graded rings, which is $R \sra R / F^{\ge 1} R$ if $(R, F)$ is tamely filtered. 
	\end{proof}
	
	Tame affine schemes over $\BA^1$ are usually not finite type. For example, let $(R, F)$ be the tamely filtered ring such that $R = k[x]$ and $F^{\ge 1}R = \la x \ra$. Then the Artin--Rees ring is $\oplus_m F^{\ge m} R \simeq k[h, h^{-1}x, h^{-2}x, \cdots]$, which is not finitely generated. Geometrically, this ring can be constructed from the affine plane $\Spec k[h, x]$ by taking an infinite sequence of affine charts of blow-ups at the origin, the first of which is $k[h, x] \to k[h, h^{-1}x]$. The following lemma tells us what finite type tame affine schemes look like: 
	
	\begin{lem} \label{tame-describe-2} 
		Let $(R, F)$ be a tamely filtered $k$-algebra. 
		\begin{enumerate}[label=(\roman*)]
			\item $\oplus_m F^{\ge m} R$ is finitely generated (as a $k$-algebra) if and only if $R$ is finitely generated and the ideal $F^{\ge 1} R$ is idempotent. 
			\item If the situation of (i) obtains, then the map $p$ from Lemma~\ref{tame-describe} looks as follows. For each connected component $C \subset \Spec R$, the map $p$ is either the identity map of $C \times \BA^1$ or the open embedding $C \times (\BA^1 \setminus \{0\}) \hra C \times \BA^1$. The former occurs if and only if the ideal $F^{\ge 1}R$ vanishes on $C$. 
		\end{enumerate}
	\end{lem}
	\begin{proof}
		(i). If $\oplus_m F^{\ge m} R$ is finitely generated, choose a finite set of homogeneous generators, and partition them into two groups $\{r_i\} \cup \{s_j\}$ where the $r_i$ live in nonpositive degrees and the $s_j$ live in positive degrees. Let $n$ be the largest degree which occurs among the generators. For degree reasons, every polynomial in the generators which lies in the summand $F^{\ge n+1} R$ must be a sum of monomials each of which contains at least two $s_j$'s. The $s_j$'s lie in $F^{\ge 1}R$ when viewed as elements of $R$ (i.e.\ forgetting the grading). Therefore 
		\[
		F^{\ge 1} R = F^{\ge n+1} R \subset (F^{\ge 1} R)^2
		\]
		as ideals of $R$. (The equality uses the tamely filtered hypothesis.) Hence $F^{\ge 1} R$ is idempotent. Also, the fiber of $\Spec \oplus_m F^{\ge m} R \xrightarrow{\pi} \Spec k[h]$ at $h=1$ is $\Spec R$, so $R$ is also finitely generated. 
		
		Conversely, assume that $F^{\ge 1} R$ is idempotent and $R$ is finitely generated. Since $R$ is Noetherian, $F^{\ge 1} R$ is a finitely generated ideal. A set of generators for $\oplus_m F^{\ge m} R$ is given by taking a set of generators for $R$ placed in degree 0, a set of ideal generators for $F^{\ge 1} R$ placed in degree $1$, and the element $1_R$ placed in degree $-1$. 
		
		Point (ii) is deduced from Lemma~\ref{tame-describe} and the fact that an idempotent ideal of a Noetherian ring vanishes on a union of connected components. 
	\end{proof}
	
	Next, we observe that every filtered ring has a universal tame quotient. 
	
	\begin{lem} \label{lem-tame-cat} 
		There are adjoint functors
		\[
		\iota : \QCoh^{\gm}_{\tf, \tame}(\BA^1) \rightleftarrows \QCoh^{\gm}_{\tf}(\BA^1) : \tau_{\tame}
		\]
		where $\iota$ is the inclusion and $\tau_{\tame}$ sends $(V, F)$ to $(F^{\ge 0} V, F_1)$ where $F_1$ is the tame filtration defined by $F_1^{\ge m}  = F^{\infty} V$ for all $m > 0$. Both of these functors are monoidal. 
	\end{lem}
	\begin{proof}
		Clear. 
	\end{proof}
	
	Let $\Sch^{\on{aff}, \tf}_{\BA^1/\gm}$ be the category of $\gm$-equivariant affine schemes flat over $\BA^1$, and let $\Sch_{\BA^1 / \gm}^{\on{aff}, \tame}$ be the subcategory of tame affine schemes. 
	
	\begin{cor}\label{cor-tame} 
	    There are adjoint functors 
		\[
		\tau^{\tame} : \Sch_{\BA^1 / \gm}^{\on{aff}} \rightleftarrows \Sch_{\BA^1 / \gm}^{\on{aff}, \tame} : \iota
		\]
		where $\iota$ is the inclusion, and both functors are monoidal with respect to the Cartesian monoidal structures. The functor $\tau^{\tame}$ preserves the class of affine schemes whose every connected component is integral and finite type. 
	\end{cor}
	\begin{proof}
		The last sentence is not obvious. It suffices to show that, if $\oplus_m F^{\ge m} R$ is finitely generated and integral, then the following statements hold: 
		\begin{enumerate}[label=(\roman*)]
			\item $F^{\ge 0} R$ is a finitely generated $k$-algebra. 
			\item The ideal $F^{\infty} R$ is either $\la 0 \ra$ or $F^{\ge 0} R$.\footnote{If $F^{\infty}R = F^{\ge 0} R$, then $F$ must be the trivial filtration on $R$, that is $F^{\ge m} R = R$ for all $m$.} 
		\end{enumerate}
		This is because of Lemma~\ref{tame-describe-2}(i) and the fact that we can consider each connected component separately. 
		
		Let $s_i$ for $i = 1, \ldots, r$ be a finite set of homogeneous generators for the graded ring $\oplus_m F^{\ge m} R$, where $F^{\ge m} R$ lives in degree $m$. Then we have 
		\[
		F^{\ge m} R = \underset{\substack{e_1, \ldots, e_r \ge 0\\ \Sigma_i\, e_i \deg(s_i) = m}}{\on{span}}  \left(\underset{i=1}{\overset{r}{\Pi}} s_i^{e_i} \right)
		\]
		where the notation means the $k$-linear span. In particular, the monomials of degree zero (whose $k$-span is $F^{\ge 0} R$) are in bijection with the kernel of the map of monoids $\BN^{\oplus r} \xrightarrow{\on{deg}} \BZ$ which sends $(e_1, \ldots, e_r) \mapsto \sum_{i} e_i \deg(s_i)$. Gordan's Lemma says that the submonoid of $\BN^{\oplus r}$ determined by the equation $\on{deg}(e_1, \ldots, e_r) = 0$ is finitely generated. This provides a finite set of $k$-algebra generators of $F^{\ge 0} R$, which proves (i). 
		
		Similarly, applying the Gordon's lemma to the submonoid of $\BN^{\oplus r}$ cut out by the inequality $\deg \ge 0$ yields a finite set of generators of the graded subring $\oplus_{m \ge 0} F^{\ge m} R$. Let use denote these generators by $t_1, \ldots, t_{\ell}$, so that $\on{deg}(t_i) \ge 0$. If all these degrees are zero, then $F^{\ge m} R = \la 0 \ra$ for $m > 0$, which proves (ii). Therefore, we may assume that not all degrees are zero. Partition the generators into two parts
		\[
		    \{t_1, \ldots, t_{\ell}\} = \{t_1, \ldots, t_{\ell_0}\} \sqcup \{t_{\ell_0+1}, \ldots, t_{\ell}\}
	    \]
	    so that $\deg(t_i) > 0$ if and only if $i \le \ell_0$. Then 
		\[
		    F^{\ge m} R = \sum_{\substack{e_1, \ldots, e_{\ell_0} \ge 0\\ \Sigma_i\, e_i \deg(t_i) = m}} \left \langle \underset{i=1}{\overset{\ell_0}{\Pi}} t_i^{e_i} \right \rangle
		\]
		where $\langle-\rangle$ denotes an ideal in $F^{\ge 0} R$. (Note that the sum ranges over possible exponents for $t_1$ through $t_{\ell_0}$. The remaining generators $t_{\ell_0 + 1}, \ldots, t_{\ell}$ appear as coefficients when one constructs the $F^{\ge 0} R$-ideal generated by the previous elements.) On the other hand, define 
		$I = \langle t_1, \ldots, t_{\ell_0} \rangle$, 
		so that 
		\[
		I^m = \sum_{\substack{e_1, \ldots, e_{\ell_0} \ge 0\\ \Sigma_i\, e_i = m}} \left \langle \underset{i=1}{\overset{\ell_0}{\Pi}} t_i^{e_i} \right \rangle. 
		\]
		Let $d := \max_i \deg(t_i) > 0$. Then we have 
		\[
		I^m \subset F^{\ge m} R \subset I^{\lfloor \frac{m}{d} \rfloor}, 
		\]
		because we have the implications 
		\[
		\Sigma_i\, e_i \ge m \quad \Longrightarrow \quad \Sigma_i\,  e_i \deg(t_i)\ge m \quad \Longrightarrow \quad \Sigma_i\, e_i d \ge m. 
		\]
		Taking intersections over all $m \ge 0$ implies that $F^{\infty} R = \cap_m I^m$. Since $R$ is integral by hypothesis, the Krull Intersection Theorem implies that $F^\infty R$ is either $\la 0 \ra$ or $F^{\ge 0}R$, which proves (ii). 
	\end{proof}
	
	We shall omit instances of $\iota$ in the subsequent notation. 
	
	\begin{cor} \label{cor-tame-univ}
		There is a finite type flat affine group scheme $\tau^{\tame}(\tilde{G}^s)$ over $\BA^1 / \gm$ and a map $q : \tilde{G}^s \to \tau^{\tame}(\tilde{G}^s)$ of group schemes over $\BA^1/\gm$ which is initial for all such maps from $\tilde{G}^s$ to a \underline{tame} affine group scheme over $\BA^1 / \gm$. The map $q$ is dominant because $q_1$ is surjective. 
	\end{cor}
	\begin{proof}
		This follows from Corollary~\ref{cor-tame}. For the last sentence, note that $q_1$ is dominant because $F^{\ge 0} R \to R$ is an injection for any filtered ring $(R, F)$, and this implies that $q_1$ is surjective because it is a map of group schemes. 
	\end{proof}
	
	In view of Lemma~\ref{tame-describe-2}(ii), $\tau^{\tame}(\tilde{G}^s)$ can be described very concretely: it is defined by a group $G^{\tame}_1$ and a subgroup $G^{\tame}_0$ corresponding to a subset of the connected components. The former is the general fiber and the latter is the special fiber of $\tau^{\tame}(\tilde{G}^s)$. There is a canonical surjective map $G^{s(1)} \sra G^{\tame}_1$ and a canonical $\gm$-equivariant map $G^{s(0)} \to G^{\tame}_0$, where the $\gm$-action on the target is trivial.
	\begin{cd}[column sep = 1.3in] 
		G^{s(1)} \ar[d, twoheadrightarrow, "{\mathrm{surjective}}"] & G^{s(0)} \ar[d, "\gm\text{-equivariant}"] \\
		G^{\tame}_1 \ar[hookleftarrow]{r}{\text{incl.\ of components}}[swap]{\text{Lemma~\ref{tame-describe}(ii)}} & G^{\tame}_0
	\end{cd}
	And $\tau^{\tame}(\tilde{G}^s)$ is the `universal' example of a family of quotients of the fibers of $\tilde{G}^s$ which admits such a description.

	\begin{rmk}
	    By the universal property of $\tau^{\tame}(\tilde{G}^s)$, Proposition~\ref{prop:linebundles}(i) is equivalent to the assertion that every character of $G^{s(1)}$ factors through the quotient map $G^{s(1)} \sra G_1^{\on{tame}}$. In addition, the character $\lim \chi$ of $G^{s(0)}$ obtained by taking the limit under specialization can be computed as the composition 
		\[
		G^{s(0)} \to G_0^{\on{tame}} \hra G_1^{\on{tame}} \xrightarrow{\chi} \gm. 
		\]
		In effect, the computation of the limiting character is encapsulated in the construction of the tame quotient, which contains information about the limits of all functions on $G^{s(1)}$ under specialization toward the zero fiber. 
	\end{rmk}

	\subsection{The subcategory associated to a tame quotient} \label{s:subcat}
	
	Let us apply~\ref{s:factor} to a map $q : \tilde{G}^s \to H$ where $H$ is tame. For example one could take $q$ to be the universal map produced by Corollary~\ref{cor-tame-univ}. We deduce the following generalization of Corollary~\ref{cor-locsys}: 
	\begin{cor}\label{cor-tame-classify}
		Let $\mc{C}_{q, (n)}(X)$ be the category whose objects are triples $(V, F, \rho)$ as follows: 
		\begin{itemize}
			\item $V$ is a vector space, $\rho : H_1 \acts V$ is a representation, and $F$ is an exhausting filtration on $V$ which is invariant under $\rho$. 
			\item We furthermore require that the restriction of the associated graded action 
			\[
			\gm \ltimes G^{s(0)} \to \gm \times H_0 \xrightarrow{\sigma} \gm \times H_1 \acts \gr^F V
			\]
			factors through the quotient $\gm \ltimes G^{s(0)} \sra G^{\ell(0)}$. Here the specialization map $\sigma$ arises from Lemma~\ref{tame-describe}(ii). 
		\end{itemize}
		The morphisms in this category are maps of filtered representations. Then there is an equivalence of categories
		\[
		\QCoh^G_{\tf, q}(X) \simeq \mc{C}_{q, (n)}(X)
		\]	
		where the left hand side is as in Definition~\ref{def-q-cat}. 
	\end{cor}
	\begin{proof}
		The proof is identical to that of Corollary~\ref{cor-locsys}. 
	\end{proof}
	
	Recall that Proposition~\ref{flat-tame-prop}(ii) gives a criterion under which the subcategory 
	\[
	    \QCoh^G_{\tf, q}(X) \subset \QCoh^G_{\tf}(X)
	\]
	can be characterized in terms of the behavior of a sheaf $\mc{F} \in \QCoh^G_{\tf}(X)$ on $U$. Thus, when this criterion is satisfied, we have a subcategory for which it is easy to detect membership and to classify objects. Unfortunately, the closer $q$ is to being `universal,' the less likely it is to satisfy this criterion. 
	
	\begin{ex} \label{ex:tame}
		In closing, we illustrate the preceding notions by constructing an example of a family of groups over $\BA^1$ which plays the role of $\tilde{G}^s$ and a tame quotient of this family which fails the criterion of Proposition~\ref{flat-tame-prop}(ii). Let $G = \GL_2$ and define the $\gm$-action on $G \times \BA^1$ via conjugation by the subgroup $\gamma : \gm \hra G$ defined by $\gamma(t) = \begin{pmatrix} 1 & 0 \\ 0 & t \end{pmatrix}$. Define the closed subscheme $\tilde{G}^s \subset G \times \BA^1$ by requiring that 
		\[
		\tilde{G}^s|_1 = \left\{ \left. \begin{pmatrix}
		a & b \\ b & a 
		\end{pmatrix}\ \right|\ a^2 - b^2 = 1 \right\}, 
		\]
		that $\tilde{G}^s$ is invariant under the $\gm$-action, and that it is flat over $\BA^1$. This implies that  
		\[
		\tilde{G}^s|_t = \left\{ \left. \begin{pmatrix}
		a & bt^{-1} \\ bt & a 
		\end{pmatrix}\ \right|\ a^2 - b^2 = 1 \right\} 
		\]
		for all $t\neq 0$, and 
		\[
		\tilde{G}^s|_0 = \left\{ \left. \begin{pmatrix}
		a & c \\ 0 & a 
		\end{pmatrix}\ \right|\ a^2 = 1 \right\}. 
		\]
		Thus, $\tilde{G}^s$ is a degeneration from the group $\gm$ to the group $\BZ / 2 \ltimes \ga$. 
		
		Next, we find the universal tame quotient $\tilde{G}^s \to \tau^{\tame}(\tilde{G}^s)$. We have $A \simeq \oh_{\tilde{G}^s|_1} \simeq k[a, b] / \la a^2 - b^2 - 1\ra$, and the filtration on $A$ is the $\la b \ra$-adic filtration, meaning that 
		\[
		F^{\ge m} A = \begin{cases}
		A & \text{ if } m \le 0 \\
		\la b \ra^m & \text{ if } m > 0. 
		\end{cases}
		\]
		Since $\cap_m \la b \ra^m = \la 0 \ra$, we have 
		\[
		F^{\ge m} \tau^{\tame} A = \begin{cases}
		A & \text{ if } m \le 0 \\
		\la 0 \ra & \text{ if } m > 0. 
		\end{cases}
		\]
		Thus $\tau^{\tame}(\wt{G}^s) \simeq \gm \times \BA^1$. The map $G^s|_1 \sra G_1^{\tame}$ is an isomorphism, while the map $G^s|_0 \to G_0^{\tame}$ is the composition
		\[
		\BZ/2 \ltimes \ga \sra \BZ / 2 \hra \gm. 
		\]
		Thus the universal tame quotient is \emph{not} flat. In fact, the only flat tame quotient of $\tilde{G}^s$ is the trivial one. 		
		
		The characterization of the filtration on $A$ as the $\la b \ra$-adic filtration allows one to say a bit more. The group scheme $\wt{G}^s$ is obtained from the constant family $\gm \times \BA^1$ via deformation to the normal cone applied to the subgroup 
		\[
		\{\pm 1\} = V(b) \hra \Spec A \simeq \gm. 
		\]
		In general, given a group $G$ and a subgroup $H$, the deformation to the normal cone of $H$ in $G$ yields a $\gm$-equivariant degeneration from $G$ to the normal bundle $\mc{N}_{H/G}$. (The group structure of the normal bundle is given by its realization as $\mc{N}_{H/G}|_1 \rtimes H$.) The degeneration from the Poincar\'e group to the Galilean group is realized in this way. 
	\end{ex}
	
	\subsection{Application to equivariant twisted \texorpdfstring{$\D$}{D}-modules} \label{apply} 
	
	Let us relate the tame quotient idea to the application discussed in Section~\ref{sec:application}. In~\ref{sec:admissibility}, we explain why the following two subcategories of $\Vect^G(X)$ are equal: 
	\begin{itemize}
		\item The subcategory consisting of sheaves $\mc{F}$ whose restriction $\mc{F}|_U$ has the property of being a strongly $G$-equivariant twisted $\D$-module with respect to a fixed $G$-equivariant twisting on $U$. 
		\item The subcategory consisting of sheaves $\mc{F}$ for which the restricted action 
		\[
			(G^{s(1)})_\circ \hra G^{s(1)} \acts \mc{F}|_{\ell(1)}
		\]
		factors through a fixed character $\chi : (G^{s(1)})_\circ \to \gm$ corresponding to the aforementioned twisting.\footnote{We emphasize that $\chi$ need not be defined on $G^{s(1)}$, but only on its identity component.}
	\end{itemize}
	We call this category $\Vect^G_\chi(X)$. This notation generalizes the notation $\Pic^G_\chi(X)$ which was introduced in~\ref{s:line}. The most important case (in view of Theorem~\ref{thm:admissibility}) is when the twisting is a power of the determinant line bundle on $U$. 
	
	In this subsection, we construct a tame quotient of $\tilde{G}^s$ corresponding to $\chi$ as above, and we reformulate the criterion of Proposition~\ref{flat-tame-prop}(ii) more concretely. The main difference between this subsection and~\ref{s:line} is that here our character $\chi$ need only be defined on the identity component of $G^{s(1)}$. 
	
	The case when $\chi$ is trivial corresponds to untwisted $\D$-modules. Since this case has already been discussed in~\ref{s:locsys}, we assume that $\chi$ is nontrivial.  
	
	\begin{lem} \label{lem-tame-chi} 
		If $\Vect^G_\chi(X)$ is nonzero, then $\ker(\chi)$ is normal in $G^{s(1)}$, not just in $(G^{s(1)})_\circ$. Let $q_1 : G^{s(1)} \to Q_1$ be the quotient by the subgroup $\ker(\chi)$. Then there is a $\gm$-equivariant map $\tilde{G}^s \to Q_1 \times \BA^1$ whose fiber at 1 equals $q_1$. 
	\end{lem}
	\begin{proof}
		The first sentence follows from the fact that, if $G^{s(1)}$ has a nonzero representation for which the action of $(G^{s(1)})_\circ$ factors through $\chi$, then the intersection of the kernel of this representation with $(G^{s(1)})_\circ$ must be $\ker(\chi)$. Thus, $\ker(\chi)$ is an intersection of two normal subgroups, so it is normal. 
		
		Now we prove the last statement. If $G^{s(1)}$ is connected, it follows from Proposition~\ref{prop:linebundles}(i). For the general case, we shall modify the proof, which uses Lemma~\ref{lem-dim-1}. Consider the codimension-one subgroup $\ker(\chi) \subset G^{s(1)}$, extend it to a $\gm$-equivariant sub group scheme $K \subset \tilde{G}^s$ which is flat over $\BA^1$, and let $\tilde{G}^s \sra \tilde{Q}$ be the quotient by $K$. Lemma~\ref{lem-dim-1}(i) applies to the identity component of $\tilde{Q}$, and it tells us that the action of $\gm$ on the identity component of $\tilde{Q}_0$ is expanding.
		
		If $(\tilde{Q}_0)_\circ \simeq \gm$, then the action $\gm \acts \tilde{Q}_0$ must be trivial (by rigidity of tori), so the proof of (i) implies (ii) in Lemma~\ref{lem-dim-1} gives a $\gm$-equivariant map $\tilde{Q} \to Q_1 \times \BA^1$, and we are done. 
		
		If $(\tilde{Q}_0)_\circ \simeq \ga$, then since every action $\gm \acts \BA^1$ is equal to a translate of a power of the standard action, and since the action $\gm \acts \tilde{Q}_0$ respects the group structure, it follows that the action $\gm \acts \tilde{Q}_0$ is expanding. Now we finish as in the previous paragraph. 
	\end{proof}
	
	\subsubsection{Definition of the tame quotient associated to $\chi$} \label{q-tame-chi} 
	
	In the notation of Lemma~\ref{lem-tame-chi}, let $Q$ be the group scheme obtained from $Q_1 \times \BA^1$ by deleting from $Q_1 \times \{0\}$ the components which are not in the image of $G^{s(0)}$. Then consider the map $q : \tilde{G}^s \to Q$ obtained by factoring the map $\tilde{G}^s \to Q_1 \times \BA^1$ through $Q$. This is a tame quotient in the sense of~\ref{sec:univtame}, and in view of Corollary~\ref{cor-tame-univ} it must factor through the universal tame quotient of $\tilde{G}^s$. 
	
	It is also related to the tame quotient of components considered in~\ref{s:locsys}. We have a map of exact sequences
	\begin{cd}
		\gm \ar[d, hookrightarrow] \ar[r, shift left = 0.5, dash] \ar[r, shift right = 0.5, dash] & \gm \ar[d, hookrightarrow] \\
		Q_1 \ar[d, twoheadrightarrow] \ar[r, hookleftarrow] & Q_0 \ar[d, twoheadrightarrow] \\
		\Com(G^{s(1)}) \ar[r, hookleftarrow] & \Im(\sigma)
	\end{cd}
	where $\sigma$ was defined in~\ref{s:locsys}. The bottom row is the `specialization map' for $H_{\locsys}$ as defined in~\ref{s:locsys}. 
	
	When the twisting is given by a power of the determinant line, there is a nice criterion for testing whether this tame quotient satisfies Proposition~\ref{flat-tame-prop}(ii). 
	
	\begin{lem} \label{lem-top}
		If the character $\chi$ is a (nonzero) rational power of the character $(G^{s(1)})_\circ \acts \wedge^{\on{top}} \mc{T}_{\ell(1)}U$, then the tame quotient $\tilde{G}^s \to Q$ constructed above is flat if and only if the action $(G^{s(0)})_{\circ} \acts \wedge^{\on{top}}\mc{T}_{\ell(0)}Z$ is nontrivial.\footnote{Since $G^{s(0)}$ acts trivially on the fibers of $\mc{N}_{Z/X}$, it does not matter whether $\mc{T}_{\ell(0)}Z$ is interpreted as the tangent space taken in $Z$ or in $X$.}  
	\end{lem}
	\begin{proof}
	    Let $\chi : G^{s(1)} \to \gm$ be the character given by acting on the line $\wedge^{\on{top}} \mc{T}_{\ell(1)}U$. First, we claim the following: 
	    \begin{itemize}
	        \item The map $G^{s(0)} \to Q_0$ is surjective if and only if the character $\lim \chi$ of $G^{s(0)}$ is nontrivial on $(G^{s(0)})_\circ$. 
	    \end{itemize}
	    Since the map $G^{s(0)} \to Q_0$ is surjective on components (by definition of $Q$), it suffices to test whether the map $(G^{s(0)})_\circ \to (Q_0)_\circ \simeq \gm$ is a surjection. This occurs if and only if it is nontrivial. Furthermore, the construction of Lemma~\ref{lem-tame-chi} shows that the latter map is given by a nonzero rational multiple of $\lim \chi$. This character is surjective if and only if $\lim \chi$ is, so the bullet point statement holds. 
	    
	    To finish, we show that $\lim \chi$ equals the character with which $G^{s(0)}$ acts on the line $\wedge^{\on{top}}\mc{T}_{\ell(0)}Z$. The line bundle $\wedge^{\on{top}}\ell^*\mc{T}_X$ on $\BA^1$ is acted on by $G^s$ and $\gm$. By~\ref{prop:linebundles-1}, this defines a (tame) quotient $q : \tilde{G}^s \to \gm \times \BA^1$. By construction, $q_1 = \chi$, so $q_0 = \lim \chi$, and the claim follows. 
	\end{proof}
	
	\begin{rmk} \label{rmk-nonflat} 
	    \begin{enumerate}[label=(\roman*)]
	        \item[]
	        \item When the criterion of Lemma~\ref{lem-top} is satisfied, then the category $\Vect^G_\chi(X)$ can be described using Corollary~\ref{cor-tame-classify}. 
	        
	        Even when the criterion is not satisfied, it is always true (by continuity) that the action $\tilde{G}^s \acts \mc{F}|_{\ell}$ corresponding to $\mc{F} \in \Vect^G_\chi(X)$ factors through the (non-tame) quotient $\tilde{G}^s \to \tilde{Q}$ constructed in the proof of Lemma~\ref{lem-tame-chi}. Using the method of Lemma~\ref{lem-dim-1}, one can prove that $\tilde{Q}$ is obtained from $Q_1 \times \BA^1$ by iterating the following two operations: 
    		\begin{itemize}
    			\item One may delete some connected components from the special fiber. 
    			\item One may apply the blow-up construction at the end of Example~\ref{ex:tame} to a (zero-dimensional) subgroup of the special fiber. 
    		\end{itemize}
    		We do not need this result, so we did not include the proof. 
    		\item By the way, the criterion is not satisfied in Examples~\ref{ex-su} and \ref{ex-sp}. But the present subsection does not apply to those examples, because in those examples the character $\chi : (G^{s(1)})_\circ \to \gm$ is trivial. In those examples, we study line bundles so we are free to use the material of~\ref{s:line} instead. For vector bundles, we could use~\ref{s:locsys}. 
	    \end{enumerate}
	\end{rmk}

	\section{Equivariant sheaves on weakly normal chains}\label{sec:wn}
	
	In this section, we explain how to classify vector bundles on \emph{weakly normal} $G$-chains. These are the $G$-chains obtained by gluing smooth simple ones (as considered in Section~\ref{sec:main}) as transversely as possible along their codimension-one orbits. 
	
	In~\ref{weak} and \ref{weak2}, we describe how to construct vector bundles on weakly normal $G$-chains by patching together vector bundles on the normalizations of their irreducible components. This material does not use Theorem~\ref{thm-main}, but that theorem (or any other classification result) can be used to describe vector bundles on the normalizations of the irreducible components. In~\ref{ssec:case-line}, we specialize to the case of line bundles and apply Theorem~\ref{thm-main} to arrive at a description of $\Vect^G(X)$ in terms of a `nearby cycles' or `gluing' functor, see Lemma~\ref{lem-nearby}. 
	
	\subsection{Weak normality} \label{weak}
	
	The condition of \emph{weak normality} was originally defined (for complex analytic spaces) by Andreotti and Norguet in \cite{AndreottiNorguet} and (for schemes) by Andreotti and Bombieri in \cite{AndreottiBombieri}. A variety $X$ is weakly normal if every finite, birational, bijective map $Y \to X$ is an isomorphism (deleting the word 'bijective' from this definition recovers the usual notion of normality). All normal schemes are weakly normal, but many others are as well. For example, a nodal cubic is weakly normal but not normal. As proved in \cite{kollar}, every Noetherian scheme $X$ admits a \emph{weak normalization}, i.e. a finite, birational map $X^{\mathrm{wn}} \to X$ satisfying the obvious universal property. 
	
	Let $G \acts X$ be a chain, let $Z \hra X$ denote the closed embedding of the disjoint union of the reduced codimension one orbits, let $\nu : \wt{X} \to X$ be the normalization, and consider the following Cartesian diagram: 
	\begin{equation} \label{patch} 
	\begin{tikzcd}
	\wt{Z} \ar[r, hookrightarrow] \ar[d] & \wt{X} \ar[d, "\nu"] \\
	Z \ar[r, hookrightarrow, "\iota"] & X
	\end{tikzcd}
	\end{equation}
	
	Since $\tilde{X}$ is a normal $G$-chain, it is smooth (see the last paragraph of~\ref{s-def-chains}). 
	
	\begin{prop}\label{weak-prop}
		If $X$ is weakly normal, then the following are true: 
		\begin{enumerate}[label=(\roman*)]
			\item $\wt{Z}$ is reduced. 
			\item The diagram~(\ref{patch}) is a Milnor square, meaning that these two equivalent conditions are satisfied: 
			\begin{itemize}
				\item (\ref{patch}) is a pushout square in the category of schemes. 
				\item The map 
				\[
				\oh_X \to \oh_{\wt{X}} \underset{\oh_{\wt{Z}}}{\times} \oh_Z
				\]
				is an isomorphism, where the terms on the right hand side are interpreted as sheaves of algebras on $X$, and the fibered product is computed in the category of such. 
			\end{itemize}
			\item The following diagram of categories is Cartesian
			\begin{cd}
				\Vect^G(X) \ar[d, "\nu^*"] \ar[r, "\iota^*"] & \Vect^G(Z) \ar[d] \\
				\Vect^G(\wt{X}) \ar[r] & \Vect^G(\wt{Z})
			\end{cd}
		\end{enumerate}
	\end{prop}
	\begin{proof}
		(i). Let $C_X \hra X$ and $C_{\wt{X}} \hra \wt{X}$ be the conductor subschemes associated to $\nu$. Proposition 53 from~\cite{kollar} implies that $C_{\wt{X}}$ and $C_X$ are reduced. Hence $C_X$ is a union of components of $Z$. The relation $C_{\wt{X}} = \nu^*C_X$ (as closed subschemes of $\wt{Z}$) implies that the components of $\wt{Z}$ which lie over $C_X$ are reduced. The map $\wt{Z} \to Z$ must be an isomorphism over the components of $Z$ which are not contained in the conductor $C_X$, so the components of $\wt{Z}$ which do not lie over $C_X$ are also reduced.  
		
		Point (ii) follows from the definition of the conductor ideal and the previous paragraph. 
		
		Point (iii) follows from Milnor's Patching Theorem, as stated in~\cite{k-book}. Although this theorem is usually stated for affine schemes (i.e.\ rings), it also applies to non-affine schemes because vector bundles are a sheaf of categories over a scheme. 
	\end{proof}

	\subsection{Patching vector bundles on chains} \label{weak2} 
	
	\subsubsection{The graph of a chain} 
	
	Preserve the notation of~\ref{weak}, so $G \acts X$ is a chain. Let $\{U_j\}_{j \in S_U}, \{Z_i\}_{i \in S_Z}, \{\tilde{Z}_k\}_{k \in S_{\tilde{Z}}}$ be the irreducible components\footnote{This term was defined in~\ref{s-def-chains}.} of $U$, $Z$, and $\tilde{Z}$, respectively. Here $S_U, S_Z, S_{\tilde{Z}}$ are sets that index the irreducible components. In this subsection we describe a diagram in the category of $G$-chains consisting of the objects $U_j, Z_i, \tilde{Z}_k$, and $U_j \cup \tilde{Z}_k$ whenever these two $G$-orbits are adjacent in $\tilde{X}$. The maps of this diagram will be given by the obvious ones shown below: 
	\begin{cd}
		& U_j \cup \tilde{Z}_k \\
		U_j \ar[ru, "\text{inclusion}"] & & \tilde{Z}_k \ar[lu, swap, "\text{inclusion}"] \ar[d, "\text{restriction of } \nu"] \\
		& & Z_i
	\end{cd}
	
	First, we define maps of sets
	\begin{cd}
	    S_U & S_{\tilde{Z}} \ar[l, swap, "a"] \ar[d, "\nu"] \\
	    & S_Z
	\end{cd}
	The map $a$ sends $k \in S_{\tilde{Z}}$ to the element $j \in S_U$ such that $\tilde{Z}_k$ is adjacent to $U_j$ in $\tilde{X}$. This $j$ is unique because $\tilde{X}$ is smooth (see~\ref{s-def-chains}). The map $\nu$ sends $k \in S_{\tilde{Z}}$ to the element $i \in S_Z$ such that the normalization map $\nu : \tilde{X} \to X$ sends $\tilde{Z}_k$ to $Z_i$. 
	
	Next, we define a directed graph $\Gamma$ as follows: 
	\begin{itemize}
	    \item The vertex set is $S_U \sqcup S_{\tilde{Z}}^{(1)} \sqcup S_{\tilde{Z}}^{(2)} \sqcup S_Z$. As indicated here, we will use superscripts $(-)^{(1)}$ and $(-)^{(2)}$ to differentiate between the two copies of $S_{\tilde{Z}}$ in the vertex set. 
	    \item For every $k \in S_{\tilde{Z}}$, there are arrows 
	    \begin{cd}
		& k^{(1)} \\
		a(k) \ar[ru] & & k^{(2)} \ar[lu] \ar[d] \\
		& & \nu(k)
	\end{cd}
	\end{itemize}
	
    Finally, we define a functor $F : \Gamma \to G\mathrm{\text{-}Chains}$ which sends $j \in S_U$ to $U_j$, sends $k^{(1)} \in S_{\tilde{Z}}^{(1)}$ to $U_{a(k)} \cup \tilde{Z}_k$, sends $k^{(2)} \in S_{\tilde{Z}}^{(2)}$ to $\tilde{Z}_k$, and sends $i \in S_Z$ to $Z_i$. The behavior on arrows is given by the first diagram in this subsubsection. 
	
	\begin{prop}\label{weak-prop-2} 
		If $X$ is weakly normal, then the limit of the diagram of categories given by $\Vect^G(-) \circ F^{\op} : \Gamma^{\op} \to \on{Cat}$ is equivalent to $\Vect^G(X)$. 
	\end{prop}
	\begin{proof}
	    For the duration of this paragraph, fix $j \in S_U$ and let $k$ range over the subset $a^{-1}(j) \subset S_{\tilde{Z}}$. The subgraph of $\Gamma$ induced by the vertices $j$ and $k^{(1)}$ is connected to the rest of $\Gamma$ only via the edges $k^{(1)} \leftarrow k^{(2)}$ for each $k$. The limit of categories associated to this induced subgraph is $\Vect^G(\ol{U}_j)$ because the $U_j \cup \tilde{Z}_k$ constitute an open cover of $\ol{U}_j$. Thus, we may modify $F$ and $\Gamma$ (without changing the limit) by contracting this subgraph to a single vertex $v_{j}$ which maps to the category $\Vect^G(\ol{U}_j)$.
		
		Now we vary $j \in S_U$ and perform this replacement for all $j \in S_U$. We obtain a new directed graph $\Gamma'$ as follows: 
		\begin{itemize}
			\item The vertex set is $\{v_j \, | \, j \in S_U\} \sqcup S_{\tilde{Z}}^{(2)} \sqcup S_{Z}$.  
			\item For each $k \in S_{\tilde{Z}}$, there are edges $v_{a(k)} \leftarrow k^{(2)} \to \nu(k)$. 
		\end{itemize} 
		We also obtain a new diagram $(\Gamma')^{\op} \to \on{Cat}$ which sends $v_j \mapsto \Vect^G(\ol{U}_j)$ for each $j \in S_U$. By construction, this diagram has the same limit as the old one. 
		
		The product of the categories associated to the $v_j$ is equivalent to $\Vect^G(\tilde{X})$. Similarly, the product of the categories associated to the $k \in S_{\tilde{Z}}$ is equivalent to $\Vect^G(\tilde{Z})$, and the product of the categories associated to the $i \in S_Z$ is equivalent to $\Vect^G(Z)$. Now it is easy to see (using the definition of $\Gamma'$) that the desired limit of categories coincides with that of Proposition~\ref{weak-prop}(iii), as desired. 
	\end{proof}
	
	\begin{rmk}
		Concretely, this means that a $G$-equivariant vector bundle on $X$ is equivalent to the following data: 
		\begin{itemize}
			\item We choose $\mc{V}_i \in \Vect^G(Z_i)$ for each $i \in S_Z$. 
			\item We choose $\mc{W}_j \in \Vect^G(U_j)$ for each $j \in S_U$. 
			\item For each $k \in S_{\tilde{Z}}$, we choose $\mc{E} \in \Vect^G(U_{a(k)} \cup \tilde{Z}_k)$ along with isomorphisms $\mc{E}|_{U_{a(k)}} \simeq \mc{W}$ and $\mc{E}|_{\tilde{Z}_k} \simeq \nu^*(\mc{V}_{\nu(k)})$, where $\nu : \tilde{Z}_k \to Z_{\nu(k)}$ is the restriction of the normalization map of $X$. 
		\end{itemize}
		Since there are no two-dimensional cells in $\Gamma'$, cocycle conditions do not appear in the above list. This is a convenient consequence of the `codimension zero or one' requirement in the definition of a chain. 
		
		This gives a purely algebraic description of $\Vect^G(X)$. Indeed, if we choose basepoints $z_i \in Z_i$ and $u_j \in U_j$, then we have equivalences $\Vect^G(Z_i) \simeq \on{Rep}(\stab_G(z_i))$ and $\Vect^G(U_j) \simeq \on{Rep}(\stab_G(u_j))$. Theorem~\ref{thm-main} can be used to describe the category $\Vect^G(U_j \cup \tilde{Z}_i)$. 
	\end{rmk}
	
	\begin{cor}
		If $X$ is weakly normal and $\Gamma$ is contractible,\footnote{This means that the topological space associated to $\Gamma$ is contractible, i.e.\ the directed graph $\Gamma$ does not contain a `zig-zag cycle.'} then the assertion of Proposition~\ref{weak-prop-2} also holds at the level of isomorphism classes. 
	\end{cor} 
	
	\subsection{The case of line bundles} \label{ssec:case-line} 
	
	\subsubsection{An analogue of the nearby cycles functor} \label{def-nearby}
	
	For this subsubsection only, we return to the setting of Section~\ref{sec:main}, so that $X = U \cup Z$ is a smooth simple fastened $G$-chain with one closed orbit. We shall describe the category $\Pic^G(X)$ in a way which more closely resembles the `gluing data' approach mentioned in~\ref{intro-gluing}. 
	
	Choose a fastening datum $(\gamma, \ell)$ for $X$. This allows one to define a `nearby cycles' functor 
	\[
	    \Psi : \Pic^G(U) \to \Pic^G(\mc{N}^{\circ}_{Z/X}), 
	\]
	given by the composition 
	\begin{cd}
			\Pic^G(U) \ar[r] & \Pic(\pt / G^{s(0)}) \ar[leftarrow]{r}{\text{restrict}}[swap]{\sim} & \Pic^G(\mc{N}^{\circ}_{Z/X})
	\end{cd}
	To define the unlabeled arrow, interpret an object of $\Pic^G(U)$ as a pair $(\chi, V)$ where $\chi : G^{s(1)} \to \gm$ is a character and $V$ is a 1-dimensional vector space. This maps to $(\lim \chi, V) \in \Pic(\pt / G^{s(0)})$ where $\lim \chi$ is the character of $G^{s(0)}$ obtained by taking a limit of $\chi$ under $\gamma$-conjugation. (Note that $\lim \chi$ exists by Proposition~\ref{prop:linebundles}(i).) 
	
	\begin{lem} \label{lem-nearby}
		We have a Cartesian diagram 
		\begin{cd}
			\Pic^G(X) \ar[r] \ar[d] & \Pic^G(Z) \ar[d, "\mathrm{pullback}"] \\
			\Pic^G(U) \ar[r, "\Psi"] & \Pic^G(\mc{N}^{\circ}_{Z/X})
		\end{cd}
	\end{lem}
	\begin{proof}
		An object in the Cartesian product of the labeled maps is given by a triple $(\chi, V, \chi')$ as follows: 
		\begin{itemize}
			\item $\chi : G^{s(1)} \to \gm$ is a character, $V$ is a 1-dimensional vector space, and $\chi' : G^{\ell(0)} \to \gm$ is a character. 
			\item These data must satisfy the requirement that the restriction of $\chi'$ along $G^{s(0)} \xrightarrow{\iota} G^{\ell(0)}$ equals $\lim \chi$. 
		\end{itemize} 
		On the other hand, an object in $\Pic^G(X)$ corresponds (via~\ref{prop:linebundles-1}) to a triple $(V, F, \chi)$ as follows: 
		\begin{itemize}
			\item $V$ is a 1-dimensional vector space, $F$ is an exhausting filtration on $V$ (i.e.\ an integer $d = \deg V$), and $\chi : G^{s(1)} \to \gm$ is a character. 
			\item These data must satisfy the requirement that the character of $\gm \ltimes G^{s(0)}$ defined by $(d, \lim \chi)$ descends along the map $\gm \ltimes G^{s(0)} \sra G^{\ell(0)}$. 
		\end{itemize}
		We define an equivalence between these two categories by sending 
		\[
			(V, F, \chi) \mapsto (\chi, V, \chi')
		\]
		where $\chi'$ is the character of $G^{\ell(0)}$ given by the last bullet point above, which in particular depends on $F$. The inverse equivalence is defined by sending $(\chi, V, \chi')$ to the triple $(V, F, \chi)$ where $F$ places $V$ in degree $d$, for the integer $d$ defined by the weight of the action $\gm \xrightarrow{\gamma} G^{\ell(0)} \acts V$. It is straightforward to check that these are mutually inverse equivalences, and that there are natural isomorphisms which make this equivalence compatible with the projections to $\Pic^G(U)$ and $\Pic^G(Z)$. For the latter, it is important to note that the vector spaces $V$ and $\gr^F V$ are \emph{canonically} isomorphic, because $V$ is 1-dimensional. 
	\end{proof}
	
	\subsubsection{The category of line bundles} 
	
	By combining Proposition~\ref{weak-prop-2} and Lemma~\ref{lem-nearby}, we obtain a characterization of $\Pic^G(X)$ as the limit of a diagram of categories with maps as follows, where $\Psi_{j, k}$ is the nearby cycles functor $\Psi$ for the smooth simple chain $U_j \cup \tilde{Z}_k$: 
	\begin{cd}
		\Pic^G(U_j) \ar[rd, "\Psi_{j, k}"] & & \Pic^G(Z_{\nu(k)}) \ar{ld}[swap]{\text{pullback}}\\
		& \Pic^G(\mc{N}^\circ_{\tilde{Z}_k / (U_j \cup \tilde{Z}_k)})
	\end{cd}
	In what follows, we explain this observation in more detail. 
	
	Let $\Gamma_{\on{pic}}$ be the directed graph defined as follows: 
	\begin{itemize}
		\item The vertex set is $S_U \sqcup S_{\tilde{Z}} \sqcup S_Z$. 
		\item For every $k \in S_{\tilde{Z}}$, there are arrows $a(k) \to k \leftarrow \nu(k)$. 
	\end{itemize} 
	Let $F_{\on{pic}} : \Gamma_{\on{pic}} \to \on{Cat}$ be the diagram defined as follows: 
	\begin{itemize}
		\item It sends $j \in S_U$ to $\Pic^G(U_j)$. 
		\item It sends $k \in S_{\tilde{Z}}$ to $\Pic^G(\mc{N}^{\circ}_{\tilde{Z}_k / (U_{a(k)} \cup \tilde{Z}_k)})$. 
		\item It sends $i \in S_Z$ to $\Pic^G(Z_i)$. 
		\item Its behavior on arrows is given by the above diagram. 
	\end{itemize} 
	
	\begin{cor}\label{cor:linebundles}
		If $X$ is weakly normal, then we have $\lim F_{\on{pic}} \simeq \Pic^G(X)$. 
	\end{cor}
	\begin{proof}
		This follows from similar `diagram modifications' as were used in the proof of Proposition~\ref{weak-prop-2}. Starting with the diagram of Proposition~\ref{weak-prop-2} but with $\Pic^G(-)$ in place of $\Vect^G(-)$, one first replaces subdiagrams of the form 
		\begin{cd}
			& \Pic^G(U_{a(k)} \cup \tilde{Z}_k) \ar[ld] \ar[rd] \\
			\Pic^G(U_{a(k)}) & & \Pic^G(\tilde{Z}_k)
		\end{cd}
		with diagrams of the form
		\begin{cd}
			\Pic^G(U_{a(k)}) \ar[rd, "\Psi_{j, k}"] & & \Pic^G(\tilde{Z}_k) \ar[ld] \\
			& \Pic^G(\mc{N}^\circ_{\tilde{Z}_k / (U_{a(k)} \cup \tilde{Z}_k)})
		\end{cd}
		Lemma~\ref{lem-nearby} tells us that this does not change the resulting limit of categories. Lastly, for each $i \in S_Z$, one may replace the subdiagram induced by $\Pic^G(Z_i)$ and $\Pic^G(\tilde{Z}_k)$ for all $k \in \nu^{-1}(i)$ by the single object $\Pic^G(Z_i)$. This is because the limit does not change if we pass to an initial subdiagram. 
	\end{proof}
	
	\begin{rmk}
		Concretely, this means that a $G$-equivariant line bundle on $X$ is equivalent to the following data: 
		\begin{itemize}
			\item We choose $\mc{V}_i \in \Pic^G(Z_i)$ for each $i \in S_Z$. 
			\item We choose $\mc{W}_j \in \Pic^G(U_j)$ for each $j \in S_U$. 
			\item For each $k \in S_{\tilde{Z}}$, we choose an isomorphism $\eta : \Psi(\mc{W}_i) \simeq \pi^*\mc{V}_{\nu(k)}$ of objects in $\Pic^G(\mc{N}^\circ_{\tilde{Z}_k / (U_{a(k)} \cup \tilde{Z}_k)})$, where $\pi$ is the composition 
			\[
				\mc{N}^\circ_{\tilde{Z}_k / (U_{a(k)} \cup \tilde{Z}_k)} \to \tilde{Z}_k \to Z_{\nu(k)}. 
			\]
		\end{itemize}
		As was the case in Proposition~\ref{weak-prop-2}, there are no cocycle conditions because $\Gamma_{\on{pic}}$ contains no 2-cells. 
	\end{rmk}

\section{An application to the representation theory of real reductive groups}\label{sec:application}

Let $G_{\mathbb{R}}$ be the real points of a connected reductive algebraic group and let $G$ be the complexification of $G_{\mathbb{R}}$. Choose a Cartan involution $\theta$ of $G$ and let $K = G^{\theta}$, the fixed-points of $\theta$. By definition, $K \cap G_{\mathbb{R}} \subset G_{\mathbb{R}}$ is a maximal compact subgroup. If we write $\mathfrak{g}$ and $\mathfrak{k}$ for the Lie algebras of $G$ and $K$, respectively, the differenital of $\theta$ provides a decomposition
$$\mathfrak{g} = \mathfrak{k} \oplus \mathfrak{p}$$
into $+1$ and $-1$ eigenspaces. 

Let $\mathcal{N} \subset \mathfrak{g}$ be the set of nilpotent elements of $\mathfrak{g}$. Then $\mathcal{N}$ is closed in the Zariski topology on $\mathfrak{g}$ and invariant under the commuting actions of $G$ and $\mathbb{C}^{\times}$. The $G$-action on $\mathcal{N}$ has finitely many orbits.

An important player in the representation theory of $G_{\mathbb{R}}$ is the closed subvariety
$$\mathcal{N}_{\theta} := \mathcal{N} \cap \mathfrak{p}$$
This subset is invariant under $K$ and $\mathbb{C}^{\times}$ (though not under $G$). By a result of Kostant-Rallis (\cite{KostantRallis1971}), the $K$-action on $\mathcal{N}_{\theta}$ has finitely-many orbits. Each $K$-orbit on $\mathcal{N}_{\theta}$ is a Lagrangian submanifold of its $G$-saturation, which is a $G$-orbit on $\mathcal{N}$.

If $V$ is an irreducible admissible representation of $G_{\mathbb{R}}$, there is an associated class in the Grothendieck group $K\mathrm{Coh}^K(\mathcal{N}_{\theta})$ of $K$-equivariant coherent sheaves on $\mathcal{N}_{\theta}$, constructed as follows. The Harish-Chandra module of $V$ is a $\mathfrak{g}$-module $M$ with an algebraic action of $K$. These structures are compatible in the two obvious ways: the action map
$$\mathfrak{g} \otimes M \to M$$
is $K$-equivariant and the $\mathfrak{g}$-action coincides on $\mathfrak{k}$ with the differentiated action of $K$. We say, in short, that $M$ is a $(\mathfrak{g},K)$-module. $M$ is irreducible, since $V$ is.

Every finite-length $(\mathfrak{g},K)$-module $M$ admits a good filtration $... \subset M_{-1} \subset M_0 \subset M_1 \subset ...$ by $K$-invariant subspaces. The associated graded $\gr(M)$ has the structure of a graded, $K$-equivariant, coherent sheaf on $\mathfrak{p}$, with support contained in $\mathcal{N}_{\theta}$. Although $\gr(M)$ depends on the filtration used to define it, its class $[\gr(M)]$ in the Grothendieck group $K\Coh^K(\mathcal{N}_{\theta})$ does not. In particular, the set
$$\mathrm{AV}(M) := \mathrm{Supp}(\gr(M)) \subset \mathcal{N}_{\theta},$$
called the \emph{associated variety} of $M$, is a well-defined invariant.\footnote{Here $\mathrm{AV}(M)$ is defined to be the \emph{reduced} support, so that the associated variety is indeed a variety. Although $\gr(M)$ may not be scheme-theoretically supported on $\mathrm{AV}(M)$ depending on the filtration of $M$, it is always possible to interpret $[\gr(M)]$ as an element of $K_0(\mathrm{AV}(M))$, and we shall do so.} The correspondence between irreducible $G_{\mathbb{R}}$-representations and classes in $K\mathrm{Coh}^K(\mathcal{N}_{\theta})$ is given by the assignment $V \mapsto [\gr(M)]$:

\begin{center}
\begin{tikzpicture}
\draw (-5,0) node[rectangle,
draw]{Irred $G_{\mathbb{R}}$-reps};
\draw (0,0) node[
rectangle,draw]{Irred $(\mathfrak{g},K)$-mods};
\draw (5,0) node[
rectangle,draw]{$K\mathrm{Coh}^K(\mathcal{N}_{\theta})$};
\draw[->] (-3.5,0) -- (-2,0);
\draw[->] (2,0) -- (3.5,0);
\draw (-2.75,.75) node[text width=3cm, align=center]{Harish-Chandra module};
\draw (2.75,.5) node{$\mathrm{gr}$};
\end{tikzpicture}
\end{center}

\vspace{4mm}

The Arthur conjectures (\cite{Arthur1983},\cite{Arthur1989}) predict the existence of a small finite set of irreducible representations of $G_{\mathbb{R}}$ -- called \emph{unipotent} by Arthur -- with very interesting properties. A working definition can be found for example in \cite{BarbaschVogan1985}. These representations have been studied extensively over the past several decades, but the general theory remains elusive. 

The unipotent representations of $G_{\mathbb{R}}$ give rise, by the correspondence described above, to a finite set of distinguished classes in $K\mathrm{Coh}^K(\mathcal{N}_{\theta})$, which we will call \emph{unipotent sheaves}. 

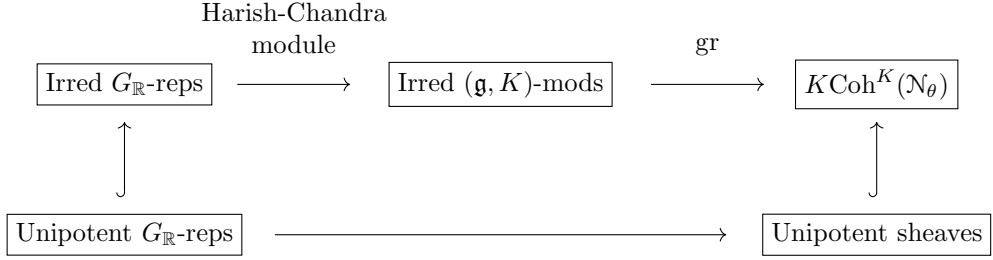
\begin{center}
\begin{tikzpicture}
\draw (-5,0) node[rectangle,
draw]{Irred $G_{\mathbb{R}}$-reps};
\draw (0,0) node[
rectangle,draw]{Irred $(\mathfrak{g},K)$-mods};
\draw (5,0) node[
rectangle,draw]{$K\mathrm{Coh}^K(\mathcal{N}_{\theta})$};
\draw[->] (-3.5,0) -- (-2,0);
\draw[->] (2,0) -- (3.5,0);
\draw (-2.75,.75) node[text width=3cm, align=center]{Harish-Chandra module};
\draw (2.75,.5) node{$\mathrm{gr}$};
\draw(-5,-2) node[rectangle, draw]{Unipotent $G_{\mathbb{R}}$-reps};
\draw(5,-2) node[rectangle, draw]{Unipotent sheaves};
\draw [right hook->] (-5,-1.5) -- (-5,-.5);
\draw [right hook->] (5,-1.5) -- (5,-.5);
\draw [->] (-3,-2) -- (3,-2);
\end{tikzpicture}
\end{center}
\vspace{4mm}

This discussion suggests the following natural question

\begin{question}\label{question}
Is there a purely geometric description of the unipotent sheaves?
\end{question}

\subsection{Admissibility} \label{sec:admissibility}

In \cite{Vogan1991}, Vogan provides a partial answer to Question \ref{question}. The key definition is the following one, due originally to Schwartz:

\begin{defn}[\cite{Schwartz}]\label{def:admissibility}
Let $U$ be a homogeneous space for $K$. Choose $e \in U$ so that $U = K/\stab_K(e)$. Homogeneous sheaves of twisted differential operators on $U$ are parameterized by characters of the Lie algebra $\mathfrak{k}^{e}$ (see, for example, Appendix A in \cite{HechtMilicicSchmidWolf}). Let $\mathrm{tr}$ be the character of $\mathfrak{k}^{e}$ corresponding to the canonical line bundle on $U$. This character is given by the explicit formula
$$\mathrm{tr}(X) = \mathrm{Tr}(\mathrm{ad}(X)|_{\mathfrak{k}^{e}}) - \mathrm{Tr}(\mathrm{ad}(X)|_{\mathfrak{k}}) \qquad X \in \mathfrak{k}^{e}$$
A $K$-equivariant vector bundle $\mathcal{E} \in \Vect^K(U)$ is \emph{admissible} if the restriction $\mathcal{E}|_{U}$ has the structure of a strongly $K$-equivariant $\mathcal{D}_{U}^{\frac{1}{2}\mathrm{tr}}$-module, where $\mathcal{D}_{U}^{\frac{1}{2}\mathrm{tr}}$ is the homogeneous sheaf of twisted differential operators corresponding to the character $\frac{1}{2}\mathrm{tr}$ of $\mathfrak{k}^{e}$.

Here is a more explicit formulation. The fiber $E = \mathcal{E}|_e$ carries a representation $\rho$ of $\stab_K(e)$. $\mathcal{E}$ is admissible if $\rho$ satisfies the condition
$$d\rho = \frac{1}{2}\mathrm{tr}\cdot \mathrm{Id}_{E} \in \mathrm{Rep}(\mathfrak{k}^e)$$
\end{defn}

As is clear from the second formulation of Definition \ref{def:admissibility}, admissibility is a property of the class $[\mathcal{E}] \in K\Vect^K(U)$.

Vogan proves

\begin{theorem}[\cite{Vogan1991}, Theorem 8.7]\label{thm:admissibility}
Suppose $M$ is the Harish-Chandra module of a unipotent representation of $G_{\mathbb{R}}$. Let $U_1,...,U_m$ be the open $K$-orbits on $\mathrm{AV}(M)$. Then the classes $[\gr(M)|_{U_i}] \in K\Vect^K(U_i)$ are admissible.
\end{theorem}

If $\mathcal{E}$ is a unipotent sheaf, Theorem \ref{thm:admissibility} imposes rigid constraints on the restrictions $\mathcal{E}|_{U_i}$ to the open $K$-orbits on its support. However, it says nothing about the relationship between the various $\mathcal{E}|_{U_i}$ or about the restrictions $\mathcal{E}|_{Z_j}$ to $K$-orbits of codimension $\geq 1$. The machinery developed in Sections \ref{sec:chains} and \ref{sec:wn} sheds light on both of these issues. 

\subsection{The chain associated to an irreducible \texorpdfstring{$(\mathfrak{g},K)$}{(g,K)}-module} \label{gk-chain} 

Let $M$ be an irreducible $(\mathfrak{g},K)$-module and let $[\gr(M)] \in K\mathrm{Coh}^K(\mathcal{N}_{\theta})$ be the class defined in Section \ref{sec:application}. If $M$ is irreducible, then $\mathrm{AV}(M)$ is of pure dimension (see \cite{Vogan1991}, Theorem 8.4). Hence, the subset
$$\mathrm{Ch}(M) := \mathrm{AV}(M) \setminus (\text{all orbits of codimension} \geq 2)$$
is a $\mathbb{C}^{\times}$-invariant $K$-chain. 

Usually, $\mathrm{Ch}(M)$ is \emph{not} fastened. The following example is typical. 

\begin{ex}\label{ex:nonfastened2}

Let $G_{\mathbb{R}} = Sp(4,\mathbb{R})$. Then $G = Sp(4,\mathbb{C})$ and $\mathfrak{g} = \mathfrak{sp}(4,\mathbb{C})$. In standard  coordinates,
$$\mathfrak{g} = \left\{\left(
\begin{array}{c|c}
A & B\\ \hline
C & -A^t
\end{array}\right): B,C \text{ symmetric}\right\}$$
Choose the Cartan involution
$$\theta: G \to G \qquad \theta(X) = \ ^tX^{-1}$$
Then $K$ is identified with $GL_2(\mathbb{C})$ and $\mathfrak{p}$ is identified (as a representation of $K$) with $\mathrm{Sym}^2\mathbb{C}^2 \oplus \mathrm{Sym}^2(\mathbb{C}^2)^*$. 

The $K$-orbits on $\mathcal{N}_{\theta}$ are parameterized by partitions of $4$ with signs attached to even parts and all odd parts occurring with even multiplicity (see \cite{CollingwoodMcgovern}). The closure orderings and codimensions are indicated below:

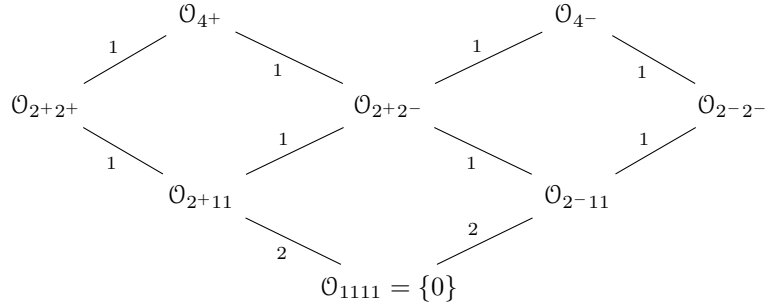
\begin{center}
	\begin{tikzcd}
& \mathcal{O}_{4^+} & & \mathcal{O}_{4^-} & \\
\mathcal{O}_{2^+2^+} \arrow[dash,ur, "1"] & & \mathcal{O}_{2^+2^-} \arrow[dash,ul, "1"] \arrow[dash,ur,"1"] & & \mathcal{O}_{2^-2^-} \arrow[dash,ul,"1"]\\
& \mathcal{O}_{2^+11} \arrow[dash,ul,"1"] \arrow[dash,ur,"1"] & & \mathcal{O}_{2^-11} \arrow[dash,ul,"1"] \arrow[dash,ur,"1"] & \\
& & \mathcal{O}_{1111} =\{0\} \arrow[dash,ul,"2"] \arrow[dash,ur,"2"]& &\\
	\end{tikzcd}
\end{center}

Choose elements $e_{4^+} \in \mathcal{O}_{4^+}$, $e_{4^-} \in \mathcal{O}_{4^-}$, and so on, and for each $e$, write $\mathrm{Red}_K(e)$ for the Levi factor of the isotropy group $\stab_K(e)$ (well-defined up to isomorphism). We compute

\begin{align*}
    &\mathrm{Red}_K(e_{4^+}) \cong \mathrm{Red}_K(e_{4^-}) \cong \{\pm 1\} \\
    &\mathrm{Red}_K(e_{2^+2^+}) \cong \mathrm{Red}_K(e_{2^-2^-}) \cong O_2(\mathbb{C}) \qquad  \mathrm{Red}_K(e_{2^+2^-}) \cong \{\pm 1\} \times \{\pm 1\}\\
    &\mathrm{Red}_K(e_{2^+11}) \cong \mathrm{Red}_K(e_{2^-11}) \cong \mathbb{C}^{\times} \times \{\pm 1\}\\
    &\mathrm{Red}_K(0) = GL_2(\mathbb{C})
\end{align*}

Let $M$ be the Harish-Chandra module of the spherical principal series representation of $G_{\mathbb{R}}$ of infinitesimal character $0$. Hence, $\mathrm{AV}(M) = \mathcal{N}_{\theta}$ and 
$$\mathrm{Ch}(M) = \mathcal{O}_{4^+} \cup \mathcal{O}_{4^-} \cup \mathcal{O}_{2^+2^+} \cup \mathcal{O}_{2^+2^-} \cup \mathcal{O}_{2^-2^-}$$
The irreducible components of $\mathrm{Ch}(M)$ are 
$$Y_+ :=  \mathcal{O}_{4^+} \cup \mathcal{O}_{2^+2^+} \cup \mathcal{O}_{2^+2^-} \qquad Y_- :=  \mathcal{O}_{4^-} \cup \mathcal{O}_{2^+2^-} \cup \mathcal{O}_{2^-2^-}$$
Both components are singular, but the normalizations $\tilde{Y}_+$ and $\tilde{Y}_-$ are smooth. Consider, for example, $\tilde{Y}_+$. It has a unique open $K$-orbit $U_+$ lying over $\mathcal{O}_{4^+}$. Over the codimension $1$ orbits in $Y_+$, $\tilde{Y}_+$ contains several closed orbits $Z_{++}^1,...,Z_{++}^m$ and $Z_{+-}^1,...,Z_{+-}^n$ lying over $\mathcal{O}_{2^+2^+}$ and $\mathcal{O}_{2^-2^-}$, respectively. If we choose $z_{+-}^1 \in Z_{+-}^1$ lying over $e_{2^+2^-}$, there is an injection
$$\mathrm{Red}_K(z_{+-}^1) \subseteq \mathrm{Red}_K(e_{2^+2^-}) \cong \{\pm 1\} \times \{\pm 1\}$$
Hence, the character $\stab_K(z_{+-}^1) \to \mathbb{C}^{\times}$ corresponding to the normal bundle of $Z_{+-}^1$ in $N(Y_+)$ factors through a finite group (of order dividing $4$). In particular, by Remark~\ref{rmk-fastened-normal}, $\mathrm{Ch}(M)$ fails to be fastened.
\end{ex}

There is a remedy for this problem which exploits the $\mathbb{C}^{\times}$-action on $\mathrm{Ch}(M)$. Almost always $\mathrm{Ch}(M)$ does not contain $0$ (in the few cases when it does, $\mathrm{Ch}(M)$ is fastened, so there is nothing to worry about). Let $\overline{\mathrm{Ch}}(M)$ be the image of $\mathrm{Ch}(M) \subset \mathfrak{p}$ in the projectivization of $\mathfrak{p}$, that is
\[
    \overline{\mathrm{Ch}}(M):=\mathrm{Ch}(M)/\mathbb{C}^{\times} \subset \mathbb{P}\mathfrak{p}.
\]
Note that $\overline{\mathrm{Ch}}(M)$ is a $K$-chain of one dimension less than $\mathrm{Ch}(M)$. In the following subsection, we will prove that $\overline{\mathrm{Ch}}(M)$ is fastened (even when $\mathrm{Ch}(M)$ is not).

\subsection{The Slodowy slice}

Let $e \in \mathcal{N}_{\theta}$. By the Jacobson-Morozov theorem, there is an $\mathfrak{sl}_2$-triple $(e,f,h)$ containing $e$ as its nilpositive element. By a result of Kostant-Rallis (\cite{KostantRallis1971}, Proposition 4), we can arrange so that $f \in \mathcal{N}_{\theta}$ and $h \in \mathfrak{k}$. 

\begin{defn}
The Slodowy slice through $e$ is the affine subspace $S_e \subset \mathfrak{p}$ defined by
$$S_e := e + \mathfrak{p}^f$$
\end{defn}

\begin{prop}\label{prop:propertiesofSlodowy}
The Slodowy slice has the following properties:
\begin{enumerate}
    \item $S_e$ is transverse to the $K$-orbit $K \cdot e \subset \mathfrak{p}$.
    \item $S_e \cap K \cdot e = \{e\}$.
    \item $S_e$ is invariant under the adjoint action of $\stab_K(e,f,h)$.
\end{enumerate}
\end{prop}

\begin{proof}
These facts are standard. Proofs can be found in \cite{GanGinzburg2001}.
\end{proof}

There is a $\mathbb{C}^{\times}$-action on $S_e$ first defined by Gan and Ginzburg in \cite{GanGinzburg2001}. 

\begin{defn}
Let $\tau: \mathbb{C}^{\times} \to K$ be the co-character defined by the requirement
$$d\tau(1) = h$$
The \emph{Kazhdan} action of $\mathbb{C}^{\times}$ on $\mathfrak{p}$ is defined by the formula
$$t \ast X = t^2\Ad(\tau(t^{-1}))(X) \qquad t \in \mathbb{C}^{\times},X \in \mathfrak{p}$$
\end{defn}

This action has the following useful properties.

\begin{prop}\label{prop:propsofKazhdanaction}
The Kazhdan action on $\mathfrak{p}$
\begin{enumerate}
    \item fixes $e$,
    \item preserves $S_e$,
    \item contracts $S_e$ onto $e$, i.e.
    \[
        \lim_{t \to 0} (t \ast X) = e \quad \forall X \in S_e,
    \]
    \item commutes with the adjoint action of $\stab_K(e,f,h)$
\end{enumerate}
\end{prop}

\begin{proof}
\begin{enumerate}
    \item Compute
    $$t \ast e = t^2\Ad(\tau(t^{-1}))(e) = t^2t^{-2}e=e$$
    \item By part (1), it suffices to show that $\mathfrak{p}^f$ is preserved by the Kazhdan action of $\mathbb{C}^{\times}$. Suppose $Y \in \mathfrak{p}^f$. Compute
    $$[f,t\ast Y] = t^2[f,\Ad(\tau(t^{-1}))(Y)] = t^2\Ad(\tau(t^{-1}))[\Ad(\tau(t))f,Y] = \Ad(\tau(t^{-1}))[f,Y]  = 0$$
    \item By part (1), it suffices to show that
    \begin{equation}\label{eqn:contracting}\lim_{t \to 0}(t \ast Y) = 0 \quad \forall Y \in \mathfrak{p}^f\end{equation}
    By the representation theory of $\mathfrak{sl}_2(\mathbb{C})$, $\tau(\mathbb{C}^{\times})$ acts on $\mathfrak{p}^f$ with nonpositive weights. Hence, the Kazhdan action on $\mathfrak{p}^f$ has strictly positive weights. Equation \ref{eqn:contracting} follows.
    \item $\stab_K(h)$ centralizes $\tau(\mathbb{C}^{\times})$ and therefore commutes in its action on $\mathfrak{p}$ with the Kazhdan action of $\mathbb{C}^{\times}$. \qedhere
\end{enumerate}
\end{proof}

The remainder of this section is devoted to proving Proposition~\ref{prop:chXfastened} which says that any $K$-chain obtained by projectivizing a $K$-chain in $\mc{N}_{\theta}$ is fastened. The proof explains how to use the Slodowy slice considered above to produce an explicit fastening datum. Since applying Theorem~\ref{thm-main} in practice requires choosing a fastening datum, the proof of Proposition~\ref{prop:chXfastened} is perhaps more useful than the statement itself. 

\begin{lem} \label{lem-slice} 
		Let $X$ be a $K$-variety which consists of an open orbit $U$ and a closed orbit $Z$. Let $Y \hra X$ be a locally closed subscheme with the following properties: 
		\begin{itemize}
			\item The codimension of $Y \cap Z \subset Z$ equals the codimension of $Y \subset X$. 
			\item The scheme-theoretic intersections $Y \cap Z$ and $Y \cap U$ are smooth. 
			\item $Y$ is Cohen--Macaulay. 
		\end{itemize}
		Then the base change of the normalization map of $X$ along $Y \hra X$ is the normalization map of $Y$. 
	\end{lem}
	\begin{proof}
		We first show that the $K$-saturation map $K \times Y \xrightarrow{a} X$ is smooth. 
		\begin{enumerate}[label=(\arabic*)]
			\item The fiber over a closed point $u \in U$ is given by the Cartesian product 
			\begin{cd}
				a^{-1}(u) \ar[r, hookrightarrow] \ar[d] & K \ar[d, "k \mapsto k^{-1} \cdot u"] \\
				Y \cap U \ar[r, hookrightarrow] & U
			\end{cd}
			The right vertical map is smooth (since the fibers are isomorphic to the stabilizer group $K^u$, which is smooth because we are working in characteristic zero), and $Y \cap U$ is smooth, so $a^{-1}(u)$ is smooth as well. 
			\item The fiber over a closed point $z \in U$ is given by the Cartesian product 
			\begin{cd}
				a^{-1}(z) \ar[r, hookrightarrow] \ar[d] & K \ar[d, "k \mapsto k^{-1} \cdot z"] \\
				Y \cap Z \ar[r, hookrightarrow] & Z
			\end{cd}
			By the same argument as in (1), we know that $a^{-1}(z)$ is smooth. 
			\item The first bullet point and the Cartesian squares in (1) and (2) imply that $a^{-1}(u)$ and $a^{-1}(z)$ have the same dimensions, for any $u$ and $z$. Because $Y$ is Cohen--Macaulay and $X$ is smooth, the Miracle Flatness Theorem implies that $a$ is flat. 
			\item The fibers of $a$ are smooth by (1) and (2). Because $a$ is also flat (by (3)) and finitely presented, it follows that $a$ is smooth. 
		\end{enumerate}
		
		Since normalization is preserved by smooth base change, the two squares in the following diagram are Cartesian: 
		\begin{cd}
			\wt{Y} \ar[r] \ar[d] & (K \times Y)^{\sim} \ar[r] \ar[d] & \wt{X} \ar[d] \\
			Y \ar[r, "1_K \times \id_Y"] & K \times Y \ar[r, "a"] & X
		\end{cd}
		where the vertical maps are normalization maps, and the upper horizontal maps arise via functoriality of normalization. (To see that the first square is Cartesian, apply the smooth base change result to the map $\pr_2 : K \times Y \to Y$ and use that $Y \xrightarrow{1_K \times \id_Y} K \times Y \xrightarrow{\pr_2} Y$ is the identity.) The outer square is Cartesian, which is the desired result. 
	\end{proof}
	
	\begin{cor}\label{cor-slice} 
		Let $X \hra \mc{N}_\theta$ be a locally closed $K$-invariant subvariety consisting of a $K$-orbit $U$ and another $K$-orbit $Z \subset \partial U$ of codimension one. Fix $e \in Z$. 
		\begin{enumerate}
			\item[(i)] For the Kazhdan action, there is a unique one-dimensional $\gm$-orbit in $S_e \cap X$ whose closure contains $e$, and this orbit is contained in $U$. 
		\end{enumerate}
		Let $C \hra S_e \cap X$ be the locally closed subvariety consisting of $e$ and the orbit in (i). 
		\begin{enumerate}
			\item[(ii)] The preimage of $C$ under the normalization map of $X$ is smooth and transverse to the reduced preimage of $Z$. 
		\end{enumerate}
	\end{cor}
	\begin{proof}
		Point (i) holds because $S_e$ is transverse to $Z \subset \mf{p}$ by Proposition~\ref{prop:propertiesofSlodowy}. Thus the (possibly singular) curve $C$ satisfies the requirements placed on $Y$ in Lemma~\ref{lem-slice}. (The second bullet point in the lemma holds because $C \cap Z$ is a reduced point, because $S_e$ is transverse to $Z$. The third bullet point holds because $C$ is a reduced curve.) The lemma implies the smoothness claim in (ii). If the preimage of $C$ is not transverse to the reduced preimage of $Z$, then pushing forward tangent vectors to $X$ shows that $\mc{T}_eC \cap \mc{T}_eZ$ is nonzero (because the map from the reduced preimage of $Z$ to $Z$ is \'etale, hence induces isomorphisms on tangent spaces), and this contradicts the fact that $S_e$ is transverse to $Z$. 
	\end{proof}

	\begin{prop}\label{prop:chXfastened}
		Let $X \subset \mathcal{N}_{\theta}$ be a $K$-chain, not containing $0$. Write $\overline{X} \subset \mathbb{P}\mathfrak{p}$ for the image of $X$ in the projectivization of $\mathfrak{p}$. Then $\overline{X}$ acquires a $K$-action via the $K$-action on $X$ and is a fastened $K$-chain with respect to that action.
	\end{prop}
	\begin{proof}
		Thanks to the structure of Definition~\ref{def:fastened}, we may assume that $X = U \cup Z$ where $U$ is a $K$-orbit and $Z \subset \partial U$ is a $K$-orbit of codimension one. In the normalization $\wt{\overline{X}} \to \overline{X}$, we choose a $K$-orbit $\wt{\overline{Z}} \subset \wt{\overline{X}}$ lying over $\overline{Z}$, and we choose a point $\bar{e}' \in \wt{\overline{Z}}$ lying over the point $\bar{e} \in \overline{Z}$ which is the image of $e \in Z$ under projectivization. It suffices to construct a fastening datum $(\gamma, \ell)$ for $(\wt{\overline{U}}, \wt{\overline{Z}})$ at $\bar{e}'$.
		
		Let $C \hra X$ be as defined in Corollary~\ref{cor-slice}. Then we have a commutative diagram in which each vertical map is a normalization map, each square is Cartesian, and $\wt{C}$ is smooth: 
		\begin{cd}
			\wt{C} \ar[r] \ar[d] & \wt{X} \ar[r] \ar[d] & \wt{\overline{X}} \ar[d] \\
			C \ar[r] & X \ar[r] & \overline{X}
		\end{cd}
		Indeed, the first square is Cartesian by Corollary~\ref{cor-slice}. The second square is Cartesian because $X \to \ol{X}$ is smooth, and normalization commutes with smooth base change. Lastly, $\wt{C}$ is smooth because the normalization of a (reduced) curve is smooth. 
		
		The whole diagram is $\gm$-equivariant with respect to the Kazhdan action, and Proposition~\ref{prop:propsofKazhdanaction}(3) implies that the action on every irreducible component of $C$ is nontrivial. In addition, the Kazhdan action on $\overline{X}$ and its normalization coincides with the ordinary action by the one-parameter subgroup $\tau(\gm) \subset K$, because the $t^2$ which appears in the definition of the Kazhdan action has no effect on the projectivization.  
		
		Since $S_e \subset \mf{p}$ is an affine subspace which does not contain $0 \in \mf{p}$, the map $S_e \to \mathbb{P}\mf{p}$ is an open embedding. Hence the map $C \to \overline{X}$ in the diagram is a locally closed embedding. Since the outer square is Cartesian, the map $\wt{C} \to \wt{\overline{X}}$ is also a locally closed embedding. Thus, the previous paragraph implies that $\wt{C} \hra \wt{\overline{X}}$ is a smooth curve which is invariant under $\tau(\gm) \subset K$, and this action is nontrivial on every irreducible component of $\wt{C}$. 
				
		Since $e \in C$, we know that $\bar{e}' \in \wt{C}$. Let $\wt{C}' \hra \wt{C}$ be the locally closed subvariety consisting of $\bar{e}'$ and the unique $\gm$-orbit adjacent to it. This curve is also invariant under $\tau(\gm)$, and the action of $\tau(\gm)$ on it is nontrivial. 
		
		Next, let us show that $\wt{C}'$ is transverse to $\wt{\overline{Z}}$. Let $\wt{Z} \subset \tilde{X}$ be the preimage of $\wt{\overline{Z}} \subset \wt{\overline{X}}$ along the map $\wt{X} \to \wt{\overline{X}}$. It is a reduced $K$-orbit which is one component of the reduced preimage of $Z \subset X$ along the normalization map $\wt{X} \to X$. If $\wt{C}' \hra \wt{\overline{X}}$ is tangent to $\wt{\overline{Z}}$, then $\wt{C}' \hra \tilde{X}$ is tangent to $\wt{Z}$. But this contradicts the transversality statement in Corollary~\ref{cor-slice}(ii). 
		
		Now $(\tau(\gm), \wt{C}')$ is the fastening datum we seek. 
	\end{proof}

\subsection{Unipotent sheaves}

We begin with a conjecture

\begin{conjecture}\label{tfconjecture}
Let $M$ be the Harish-Chandra module of a unipotent representation of $G_{\mathbb{R}}$ and let $U_1,...,U_m$ be the open $K$-orbits on $\mathrm{AV}(M)$. Then

\begin{enumerate}
    \item The classes $[\gr(M)|_{U_i}]$ are irreducible.
    \item There is a good filtration of $M$ such that $\gr(M)$ is a torsion-free, scheme-theoretically supported on $\mathrm{AV}(M)$, and
    \item $\gr(M) = j_*j^*\gr(M)$, where $j: \mathrm{Ch}(M) \subset \mathrm{AV}(M)$ is the inclusion.
\end{enumerate}
\end{conjecture}

The first claim is a conjecture of Vogan in \cite{Vogan1991}.
Evidence for the second and third claims comes from \cite{Vogan1991}, \cite{MasonBrown2018}, and many low-rank examples (including the two given below).

If we assume Conjecture \ref{tfconjecture}(2), then $\gr(M)$ restricts to a torsion-free, $K \times \gm$-equivariant coherent sheaf on $\mathrm{Ch}(M)$. If $G_{\mathbb{R}}$ is classical, then the component groups $\Com(\stab_K(e_i))$ are abelian. Hence, Conjecture \ref{tfconjecture}(1) combined with Theorem \ref{thm:admissibility} implies that $\gr(M)|_{\mathrm{Ch}(M)}$ is a $K \times \gm$-equivariant line bundle. This line bundle descends along the projection map to a $K$-equivariant line bundle on $\overline{\mathrm{Ch}}(M)$. Since $\overline{\mathrm{Ch}}(M)$ is fastened (by Proposition \ref{prop:chXfastened}), we can describe this line bundle using Proposition \ref{prop:linebundles}.

\begin{ex} \label{ex-su}
Let $G_{\mathbb{R}} = SU(1,1)$. Define $\theta$ on $G= SL_2(\mathbb{C})$ by
$$D = \begin{pmatrix}1 & 0 \\ 0 & -1 \end{pmatrix} \qquad \theta(X) = DXD^{-1}$$
Then
$$K = \left\{ \mathrm{diag}(t,t^{-1}): t \in \mathbb{C}^{\times} \right\} \qquad \mathcal{N}_{\theta} = \left\{ \begin{pmatrix}0 & a \\ b & 0\end{pmatrix}: ab=0 \right\}$$
There are three $K$-orbits on $\mathcal{N}_{\theta}$: two non-zero $K$-orbits
$$\mathcal{O}_+:= \left\{ \begin{pmatrix}0 & a \\ 0 & 0\end{pmatrix}: a \neq 0 \right\} \qquad \mathcal{O}_-:= \left\{ \begin{pmatrix}0 & 0 \\ b & 0\end{pmatrix}: b \neq 0 \right\},$$
each isomorphic (as homogeneous spaces) to $K/\{\pm 1\}$, and the zero orbit $\{0\} \cong K/K$. There are four chains in $\mathcal{N}_{\theta}$:

$$\overline{\mathcal{O}_+} = \mathcal{O}_+ \cup \{0\} \qquad \overline{\mathcal{O}_-} = \mathcal{O}_- \cup \{0\} \qquad \mathcal{N}_{\theta} = \mathcal{O}_+ \cup \mathcal{O}_- \cup \{0\} \qquad \{0\}$$

All four $K$-chains are weakly normal and fastened. By Corollary \ref{cor:linebundles}, a $K$-equivariant line bundle on $\mathcal{N}_{\theta}$ is specified (up to isomorphism) by a tuple $(\rho_+,\rho_-,\lambda)$ consisting of two characters $\rho_+$ and $\rho_-$ of $\{\pm 1\}$ (defining line bundles on $\mathcal{O}_+$ and $\mathcal{O}_-$) and a character $\lambda$ of $\mathbb{C}^{\times}$ (defining a line bundle on $\{0\}$). The gluing condition of that Corollary is the requirement that $\rho_+ = \lambda|_{\{\pm 1\}} = \rho_-$. Similar statements can be formulated for the other three $K$-chains:

\begin{center}
\begin{tabular}{|l|l|} \hline
   Chain  & $(\rho_i,\lambda)$ \\ \hline
   $\{0\}$  & $(n)$ \ $n \in \mathbb{Z}$ \\ \hline
   $\mathcal{O}_+ \cup \{0\}$ & $(\epsilon, n)$ \ $\epsilon \in \mathbb{Z}/2\mathbb{Z}, n \in \mathbb{Z}, n \equiv \epsilon \mod{2}$ \\ \hline
   $\mathcal{O}_- \cup \{0\}$ & $(\epsilon, n)$ \ $\epsilon \in \mathbb{Z}/2\mathbb{Z}, n \in \mathbb{Z}, n \equiv \epsilon \mod{2}$ \\ \hline
   $\mathcal{O}_+ \cup \mathcal{O}_- \cup \{0\}$ & $(\epsilon,\epsilon,n)$ \ $\epsilon \in \mathbb{Z}/2\mathbb{Z}, n \in \mathbb{Z}, n \equiv \epsilon \mod{2}$ \\ \hline
\end{tabular}
\end{center}

There are four unipotent representations of $SU(1,1)$ and they correspond to the line bundles $(\rho_0) = (0), (\rho_+,\lambda_0) = (1,1), (\rho_-,\lambda_0) = (1,-1)$, and $(\rho_+,\rho_-,\lambda_0) = (0,0,0)$. They are, respectively, the trivial representation, the two limit of discrete series representations and the spherical principal series representation of infinitesimal character $0$.

\end{ex}

\begin{ex} \label{ex-sp}
Return to the setting of Example~\ref{ex:nonfastened2}. There are two unipotent representations of $Sp(4,\mathbb{R})$ (of infinitesimal character $(1,0)$) attached to the orbit $\mathcal{O}_{2^+2^+}$. The associated chain is $\mathcal{O}_{2^+2^+} \cup \mathcal{O}_{2^+11}$. It is smooth and fastened. By Corollary \ref{cor:linebundles}, a $K$-equivariant line bundle on $\mathcal{O}_{2^+2^+} \cup \mathcal{O}_{2^+11}$ is specified (up to isomorphism) by a pair $(\rho,\lambda)$ consisting of a character $\rho$ of $\stab_K(e_{2^+2^+}) = O_2(\mathbb{C})$ and a compatible character $\lambda$ of 
$$\stab_K(e_{2^+11}) = \left\{ \begin{pmatrix}\pm 1 & \ast \\ 0 & \ast \end{pmatrix} \right\}$$
A character of $O_2(\mathbb{C})$ is given by $\epsilon \in \mathbb{Z}/2\mathbb{Z}$ (either $\triv$ or $\det$). A character of $\stab_K(e_{2^+11})$ is given by $\mu \otimes m$ for $\mu \in \mathbb{Z}/2\mathbb{Z}$ and $m \in \mathbb{Z}$. The compatibile pairs $(\rho,\lambda)$ are as follows
$$(\rho_{2^+2^+},\lambda_{2^+11}) = (\epsilon, m \otimes \mu) \qquad \epsilon \equiv \mu \mod{2}$$
The unipotent representations correspond to the line bundles given by $(0,2 \otimes 0)$ and $(1, 2 \otimes 1)$. 
\end{ex}

\bibliographystyle{plain}
\bibliography{bibliography.bib}
\end{document}